\documentclass[final]{amsart}

\usepackage{times}
\usepackage{color}
\usepackage{amsfonts}
\usepackage{graphicx}
\usepackage{hyperref}
\usepackage{verbatim}
\usepackage{subfigure}

\usepackage{amsmath}
\usepackage{amsthm}
\usepackage{amssymb}

\newenvironment{enumerateX}
{\begin{list}{\arabic{enumi})}{
\usecounter{enumi}
\leftmargin 2.5em\topsep 0.5em\itemsep -0.0em\labelwidth 50.0em}}
{\end{list}}

\newenvironment{itemizeX}
{\begin{list}{\labelitemi}{
\leftmargin 2.5em\topsep 0.5em\itemsep -0.0em\labelwidth 50.0em}}
{\end{list}}

\newtheorem{theorem}{Theorem}[section]
\newtheorem{corollary}[theorem]{Corollary}
\newtheorem{lemma}[theorem]{Lemma}

\newtheorem{assumption}[theorem]{Assumption}
\newtheorem{proposition}[theorem]{Proposition}

\newtheorem{example}[theorem]{Example}
\newtheorem{remark}[theorem]{Remark}

\numberwithin{equation}{section}

  \newcounter{mnote}
  \let\oldmarginpar\marginpar
    \renewcommand\marginpar[1]{\-\oldmarginpar[\raggedleft\footnotesize #1]%
    {\raggedright\footnotesize #1}}
\newcommand{\beq}{\begin{equation}}
\newcommand{\eeq}{\end{equation}}
\newcommand{\beqa}{\begin{eqnarray}}
\newcommand{\eeqa}{\end{eqnarray}}
\renewcommand{\div}{{\operatorname{div}}}

\newcommand{\N}{{\mathbb N}}       
\newcommand{\PP}{{\mathbb P}}       
\newcommand{\R}{{\mathbb R}}       

\newcommand{\V}{{\mathbb V}}

\newcommand{\cB}{{\mathcal B}}

\newcommand{\cD}{{\mathcal D}}
\newcommand{\cE}{{\mathcal E}}

\newcommand{\cI}{{\mathcal I}}

\newcommand{\cL}{{\mathcal L}}
\newcommand{\cM}{{\mathcal M}}
\newcommand{\cN}{{\mathcal N}}

\newcommand{\cS}{{\mathcal S}}
\newcommand{\cT}{{\mathcal T}}

\newcommand{\bD}{{\bf D}}

\newcommand{\half}{{\scriptstyle 1/2}}

\DeclareMathAlphabet{\mathpzc}{OT1}{pzc}{m}{it}

\newcommand{\bit}{\begin{itemize}}
\newcommand{\eit}{\end{itemize}}
\newcommand{\iic}{\emph}

\newcommand{\f}{\frac}

\newcommand{\an}{\text{ and }}

\newcommand{\st}{\text{\ such that }}

\newcommand{\tforall}{\text{ for all }}

\newcommand{\rest}{\big|}

\newcommand{\forma}{a(\, \cdot \, , \, \cdot \,)}

\newcommand{\goto}{\rightarrow}

\newcommand{\lbb}{\llbracket}
\newcommand{\rbb}{\rrbracket}

\newcommand{\norm}[1]{\ensuremath{\lVert{#1} \rVert}}

\newcommand{\nr}[1]{\norm{#1}} 

\newcommand{\nrse}[1]{| \! | \! | {#1} | \! | \! |} 

\newcommand{\pa}{\partial}

\newenvironment{qqq}{\begin{eqnarray*}\begin{split}\end{split}}{\end{eqnarray*}}
\newcommand\bqq{\begin{qqq}}
\newcommand\eqq{\end{qqq}}

\newenvironment{dsub}[2]{  \begin{array}{ccccccccccccccc}{#1} \\ {#2}}{\end{array} }
\newcommand\bml{\begin{dsub}}
\newcommand\eml{\end{dsub}}

\newenvironment{mat}{\left(\begin{array}{ccccccccccccccc}}{\end{array}\right)}
\newcommand\bcm{\begin{mat}}
\newcommand\ecm{\end{mat}}

\newenvironment{rmat}{\left(\begin{array}{rrrrrrrrrrrrr}}{\end{array}\right)}
\newcommand\brm{\begin{rmat}}
\newcommand\erm{\end{rmat}}

\usepackage{mathtools}
\usepackage{stmaryrd}

\begin{document}

\title[GOAFEM for Semilinear PDE]{Convergence of Goal-Oriented Adaptive Finite Element Methods  for Semilinear Problems}

\author[M. Holst]{Michael Holst}
\email{mholst@math.ucsd.edu}
\address{Department of Mathematics, 
University of California San Diego,   La Jolla, CA 92093.}

\author[S. Pollock]{Sara Pollock}
\email{snpolloc@math.tamu.edu}
\address{Department of Mathematics, 
Texas A\&M University,   College Station, TX 77843.}

\author[Y. Zhu]{Yunrong Zhu}
\email{zhuyunr@isu.edu}
\address{Department of Mathematics, 
Idaho State University,   Pocatello, ID 83209-8085.}

\thanks{
MH was supported in part by NSF Awards~1065972, 1217175, 1262982, 1318480, 
and by AFOSR Award FA9550-12-1-0046.
SP and YZ were supported in part by NSF Awards~1065972 and 1217175.
YZ was also supported in part by NSF DMS 1319110, and in part by University Research Committee Grant No. F119 at Idaho State University, Pocatello, Idaho.
}



\begin{abstract}
In this article we develop a convergence theory for goal-oriented adaptive 
finite element algorithms designed for a class of second-order 
semilinear elliptic equations.
We briefly discuss the target problem class, and introduce several 
related approximate dual problems that are crucial to both the analysis 
as well as to the development of a practical numerical method.
We then review some standard facts concerning conforming finite element 
discretization and error-estimate-driven adaptive finite element methods
(AFEM).
We include a brief summary of \emph{a priori} estimates for this
class of semilinear problems, and then describe some goal-oriented 
variations of the standard approach to AFEM (GOAFEM).
Following the recent approach of Mommer-Stevenson and Holst-Pollock
for increasingly general linear problems, we first establish a 
quasi-error contraction result for the primal problem.
We then develop some additional estimates that make it possible
to establish contraction of the combined primal-dual quasi-error, 
and subsequently show convergence with respect to the quantity of interest.
Finally, a sequence of numerical experiments are then carefully examined.
It is observed that the behavior of the implementation follows the 
predictions of the theory.
\end{abstract}

\maketitle

\section{Introduction}
\label{sec:intro}

In this article we develop convergence theory for a class of 
goal-oriented adaptive finite element methods for 
second order semilinear equations.
In particular, we establish strong contraction results 
for a method of this type for the problem: 
\begin{equation}
\label{eqn:strong}
\left\{\begin{array}{rl}
	\cN(u):=-\nabla \cdot (A \nabla u) + b(u)  = f, &\mbox{ in }\; \Omega, \\
	u = 0,  &\mbox{ on } \;\partial \Omega,
	\end{array}
\right.
\end{equation}
with $f \in L_2(\Omega)$ and $\Omega \subset \R^d$ ($d=2$ or 3) a polyhedral domain.  
We consider the problem with $A\colon \Omega \goto \mathbb{R}^{d \times d}$  
Lipschitz and symmetric positive definite (SPD).
The standard weak formulation of the primal problem reads:  Find $u \in H^1_0(\Omega)$ such that
\begin{align}
\langle \cN(u), v\rangle:=a(u,v) + \langle b(u),v\rangle &= f(v), \quad \forall v \in H_0^1(\Omega),
\label{primal_problem}
\end{align}
where
\begin{equation}
a(u,v) = \int_{\Omega} A \nabla u \cdot \nabla v ~dx.
\label{primal_forms}
\end{equation}

In many practical applications, one is more interested in certain physically relevant aspects of the solution, referred to as ``quantities of interest'', such as (weighted) averages, flow rates or velocities. These quantities of interest are often characterized by the value $ g(u)$, where $u$ is the solution of \eqref{eqn:strong} and $g$ is a  linear functional associated with a particular ``goal''. Given a numerical approximation $u_{h}$ to the solution $u$, goal-oriented error estimates use duality techniques rather than the energy norm alone  to estimate the error in the quantity of interest . The solution of the dual problem can be interpreted as the generalized Green's function, or the \emph{influence function} with respect to the linear functional, which often quantifies the stability properties of the computed solution. There is a substantial existing literature on developing reliable and accurate \emph{a posteriori} error estimators for goal-oriented adaptivity; see \cite{Eriksson.K;Estep.D;Hansbo.P1995,Becker.R;Rannacher.R1996,Becker.R;Rannacher.R2001,EHM01,Oden.J;Prudhomme.S2001,EHL02,Giles.M;Suli.E2003,Gratsch.T;Bathe.K2005,Korotov.S2006} and the references cited therein.
To our knowledge, the results presented here are the first to show convergence in the sense of the goal function for the class of semilinear elliptic problems discussed below.   We support our theory with a numerical comparison of our method with standard goal-oriented adaptive strategies, demonstrating comparable efficiency with the added benefit of provable contraction for this problem class.

Our focus in this paper is on developing a goal-oriented adaptive algorithm for semilinear
problems \eqref{primal_problem} along with a corresponding strong contraction result, following 
the recent approach in~\cite{MoSt09,HoPo11a} for linear problems. One of the main challenges 
in the nonlinear problem that we do not see in the linear case is the dependence of the dual problem on the primal solution $u$. As it is not practical to work with a dual problem we cannot accurately form, we  develop a method for semilinear problems in which adaptive 
mesh refinement is driven both by residual-based approximation to the error 
in $u$, and by a sequence of \emph{approximate dual problems} which only depend on the numerical solution from the previous step. While globally reducing the error in the primal problem necessarily yields 
a good approximation to the goal error  $g(u- u_{h})$, methods of the type we describe 
here bias the error reduction in the direction of the goal-function $g$ 
in the interest of achieving an accurate approximation in fewer 
adaptive iterations.

Contraction of the adaptive finite element algorithm for the (primal) semilinear problem \eqref{primal_problem} has been established in~\cite{HTZ09a} and~\cite{Holst.M;McCammon.J;Yu.Z;Zhou.Y2012}.  Here we recall the contraction argument for the primal problem and use a generalization of this technique to establish the contraction of a linear combination of the primal and limiting dual quasi-errors by means of a computable sequence of approximate dual problems.  We relate this result to a bound on the error in the quantity of interest. Following~\cite{HTZ09a}, the contraction argument follows from first establishing three preliminary results for two successive AFEM approximations $u_1$ and $u_2$, and respectively $\hat z_1 \an \hat z_2$ of the primal and limiting dual problems (see Section~\ref{sec:setup} for detailed definitions). 
\begin{enumerateX}
\item Quasi-orthogonality: There exists $ \Lambda_G > 1$ such that 
\[
\nrse{u - u_2}^2 \le  \Lambda_G \nrse{u - u_1}^2 - \nrse{u_2 - u_1}^2.
\]

\item Error estimator as upper bound on error: There exists $C_1 > 0$ such that
\[
\nrse{u - u_k}^2 \le C_1 \eta_k^2(u_k, \cT_k), \quad k = 1,2.
\]
\item Estimator reduction: For $\cM$ the marked set that takes refinement $\cT_1 \goto \cT_2$, for positive constants $\lambda < 1 \an \Lambda_1$ and any $\delta > 0$ 
\[
\eta_2^2(v_2,\cT_2) \le (1 + \delta) \{ \eta_1^2(v_1 , \cT_1) - \lambda \eta_1^2(v_1, \cM) \} + 
(1 + \delta^{-1}) \Lambda_1 \eta_0^2 \nrse{v_2 - v_1}.
\]
\end{enumerateX}
For the primal problem, the mesh at each iteration may be marked for refinement with respect to the error indicators following the D\"orfler marking strategy (cf. \cite{Dorfler.W1996}).  In the case of the dual problem, the limiting estimator as used in the contraction argument is related to a computable quantity.  This quantity is the dual estimator, based on the residual of the approximate dual sequence.  The mesh is marked for refinement with respect to this set of error indicators, which correspond to the approximate dual problem at each iteration.  The transformation between limiting and approximate dual estimators couples the contraction of error in the limiting dual to the primal problem.  The final result is the contraction of what we refer to here as the \emph{combined quasi-error}
\[
\bar Q^2(u_j, \hat z_j)  \coloneqq \nrse{\hat z - \hat z_j}^2  + \gamma \zeta_2^2(\hat z_j) + \pi\nrse{u - u_j}^2 + \pi \gamma_p \eta_2^2(u_j), 
\]
which is
the sum of the quasi-error as in~\cite{CKNS08} for the limiting dual problem and a multiple of the quasi-error for the primal problem.  The contraction of this property as shown in Theorem~\ref{quasi_combine_C} establishes the contraction of the error in the goal function as shown in Corollary~\ref{goalErrorBound}.

Our analysis is based on the recent development in the contraction framework for semilinear 
and more general nonlinear problems in~\cite{HTZ09a,Holst.M;McCammon.J;Yu.Z;Zhou.Y2012,Holst.M;Szypowski.R;Zhu.Y2013},
and those for linear problems developed by
Cascon, Kreuzer, Nochetto and Siebert~\cite{CKNS08}, 
and by Nochetto, Siebert, and Veeser~\cite{NSV09}.
In addressing the goal-oriented problem we base our framework on that of 
Mommer and Stevenson~\cite{MoSt09} for symmetric linear problems and 
Holst and Pollock~\cite{HoPo11a} for nonsymmetric problems.
We note also two other recent convergence results in the literature for 
goal-oriented adaptive methods applied to self-adjoint linear
problems, namely~\cite{DKV06} and~\cite{MSST06},
both providing convergence rates in agreement with those in~\cite{MoSt09}.  

The analysis of the goal-oriented method for nonlinear problems is
significantly more complex than the previous analysis for linear
problems in~\cite{MoSt09,HoPo11a}.   We follow a marking strategy similar to the one discussed 
in~\cite{HoPo11a}; in particular, we mark for both primal and dual problems and take
the union of the two as our marked set for the next refinement.  This strategy differs from that in~\cite{MoSt09} in which they choose the set of lesser cardinality and use this to develop a quasi-optimal complexity result for solving Poisson's equation.   Due to the increased complexity of the problems we consider here, we show convergence with respect to the quasi-error as opposed to the energy error and as such mark for both primal and dual sets as the error estimator is not guaranteed to decrease monotonically for the dual problem if the mesh is only marked for the primal (and vice-versa).  While we do not develop theoretical complexity results for this method, we demonstrate it efficiency numerically and see that it compares well with the method of~\cite{MoSt09} as well as the dual weighted residual (DWR) method.
The analysis further departs from that in~\cite{HoPo11a} as here we are faced with analyzing linearized and approximate dual sequences 
as opposed to a single dual problem in order to establish contraction with 
respect to the quantity of interest.
The approach presented here allows us to establish a contraction result for the 
goal-oriented method, which appears to be the first result of this type 
for nonlinear problems. 

{\bf\em Outline of the paper.}
The remainder of the paper is structured as follows.
In \S\ref{sec:setup},
we introduce the approximate, linearized and limiting dual problems.  We briefly discuss
the problem class and review some standard
facts concerning conforming finite element discretization and
error-estimate-driven adaptive finite element methods (AFEM).  
In \S\ref{subsec:fem} we include a brief summary of \emph{a priori} estimates for the semilinear problem.
In \S\ref{sec:goafem}, we describe a goal-oriented variation of the standard
approach to AFEM (GOAFEM).
In \S\ref{sec:primalC} we discuss contraction theorems for the primal problem.  
In \S\ref{sec:contraction_thms} we introduce additional estimates necessary for the contraction of the combined quasi-error  and convergence in the sense of the quantity of interest.
Lastly, in \S\ref{sec:num} we present some numerical experiments
that support our theoretical results.

\section{Preliminaries}
   \label{sec:setup}
In this section, we state both the (nonlinear) primal problem and its finite element discretization. We then introduce the linearized dual problem, and consider some variants of this problem which are of use in the subsequent computation and analysis.

Consider the semilinear problem~\eqref{primal_problem}, 
where as in~\eqref{primal_forms} we define the bilinear form 
\[
a(u, v) = ( A \nabla u , \nabla v ),
\]
with
$( \cdot, \cdot )$ denoting the $L_2$ inner-product over 
$\Omega \subset \R^d$.  We make the following assumptions on the data:
\begin{assumption}[Problem data]\label{data_assumptions}
The problem data $\bD= (A,b,f)$ satisfies 
\begin{enumerateX}
\item $A: \Omega \goto \R^{d \times d}$ is Lipschitz continuous and symmetric positive-definite with
\begin{align*}
\inf_{x \in \Omega} \lambda_{\text{min}}(A(x)) &= \mu_0 > 0,\\
\sup_{x \in \Omega} \lambda_{\text{max}}(A(x)) &= \mu_1 < \infty.
\end{align*}
\item $b: \Omega \times \R \goto \R$ is smooth on the second argument. Here and in the remainder of the paper, we write $b(u)$ instead of $b(x,u)$ for simplicity. Moreover, we assume that $b$ is monotone (increasing):
\[
b'(\xi) \ge 0, ~\tforall \xi \in \R.
\] 
 
\item $f\in L_2(\Omega)$.
\end{enumerateX}
\end{assumption}

The native norm is the Sobolev $H^1$ norm given by $\nr v_{H^1}^2 =  ( \nabla v, \nabla v ) + ( v,v ).$ Continuity of $\forma$ follows from the H\"older inequality, 
\begin{align}\label{continuity} 
a(u, v)  \le  \mu_1  |u|_{H^1}  |v|_{H^1}
 = M_\cE \nr{u}_{H^1} \nr v_{H^1} \quad (\mbox{ with } M_\cE = \mu_{1}).
\end{align}
Define the energy semi-norm by the principal part of the differential operator
$
\nrse v^2 \coloneqq a(v,v).
$
The coercivity of $\forma$ follows from the Poincar\'e inequality with constant $C_\Omega$
\begin{align}\label{coercive}
a(v,v) & \ge \mu_0|v|_{H^1}^2 
 \ge C_\Omega \mu_0 \nr v_{H^1}^2 = m_{\cE}^2\nr v_{H^1}^2,
\end{align}
which establishes the energy semi-norm as a norm.
Putting this together with~\eqref{continuity}
establishes the equivalence between the native and energy norms.

\subsection{Linearized dual problems}
\label{subsec:linearized_dual}
Given a linear functional $g\in L_{2}(\Omega)$, the objective in goal-oriented error estimation is to relate the residual to the error in the quantity of interest. This involves solving a dual problem whose solution $z$ satisfies the relation $g(u - u_h) = \langle R(u_h), z \rangle$. In the linear case, the appropriate dual problem is the formal adjoint of the primal  (cf.~\cite{MeNo05,HoPo11a}). For $b$ nonlinear, the primal problem \eqref{primal_problem} does not have an exact formal adjoint.  In this case we obtain the dual by linearization.  

Formally, given a numerical approximation $u_{j}$ to the exact solution $u$, the residual is given by 
$$
	R(u_{j}):= f- \cN(u_{j}) = \cN(u) - \cN(u_{j}).
$$
If $z^{j}\in H_{0}^{1}(\Omega)$ solves the following linearized dual problem 
\begin{align}\label{linearized_dual}
a(z^j,v) + \langle \cB_j z^j,v\rangle &=  g(v), \quad \forall v \in H_0^1(\Omega),
\end{align}
where $g(v):=\int_{\Omega} gv dx$ and the operator $\cB_j$ is given by
\begin{equation}\label{dual_op}
\cB_j \coloneqq \int_0^1 b'(\xi u + (1 - \xi) u_j ) ~d \xi = \int_0^1 b'(u_j + (u - u_j)\xi) ~d\xi,
\end{equation}
then the goal-oriented error $g(e_j)$ of $e_{j} = u - u_{j}$ can be represented exactly by the inner product of $z^{j}$ and $R(u_{j})$: 
$$
	g(e_{j}) = \langle  R(u_{j}), z^{j} \rangle.
$$
In fact, by definition of the residual $R(u_{j})$, we have 
\begin{eqnarray*}
	\langle R(u_{j}), z^{j}\rangle = a(z^{j}, e_{j}) + \langle z^{j}, b(u) - b(u_{j})\rangle = a(z^{j}, e_{j}) + \langle \cB_j z^j, e_j \rangle = g(e_{j}).
\end{eqnarray*}
Here we used the integral Taylor identity:
$$b(u)-b(u_{j}) = \int_0^1 b'(u_j + (u - u_j)\xi) ~d\xi (u-u_{j}) =\cB_{j} (u-u_{j}).$$
The derivation and numerical use of the linearized dual problem is further discussed in~\cite{ELW00,EHM01,Holst03}.

Unfortunately, the dual problem \eqref{linearized_dual} is computationally useful because the operator $\cB_{j}$ depends on the exact solution $u$. In order to define a computable dual operator, we introduce the approximate operator $b'(u_j)$, which lead to the following approximate dual problem: 
Find $\hat z^j \in H^1_0(\Omega)$ such that
\begin{align}\label{Eapprox_dual}
a(\hat z^j,v) + \langle b'(u_j) \hat z^j,v\rangle &= g(v), \quad \forall v \in H_0^1(\Omega).
\end{align} 
The equation \eqref{Eapprox_dual} is instrumental for defining a computable \emph{a posteriori} error indicator for the dual problem.  

A further difficulty arises in the analysis of the goal-oriented adaptive algorithm driven by the \emph{a posteriori} error estimators for the approximate dual problem \eqref{Eapprox_dual}. Due to the dependence on $u_{j}$, \eqref{Eapprox_dual} changes at each step of the adaptive algorithm. This is one of the essential differences of the nonlinear problem as compared to the linear cases in the previous literature (cf. \cite{MoSt09,HoPo11a}). To handle this obstacle, we introduce the limiting  dual problem: 
Find $\hat z \in H^1_0(\Omega)$ such that
\begin{align}\label{limiting_dual_problem}
a(\hat z,v) + \langle b'(u) \hat z ,v\rangle &=g(v), \quad \forall v \in H_0^1(\Omega).
\end{align}
While the operator $b'(u)$ is a function of the exact solution $u$ and is not a computable quantity, it is the operator used in the limit of both the linearized dual \eqref{linearized_dual} and approximate dual problems \eqref{Eapprox_dual} as $u_j \goto u$.  Therefore, both the linearized and approximate sequences approach the same limiting problem \eqref{limiting_dual_problem}. Our contraction result in Theorem~\ref{quasi_combine_C} is written with respect to the limiting dual problem as defined by the operator $b'(u)$.


\subsection{Finite Element Approximation}\label{subsec:fem}
For a given conforming, shape-regular triangulation $\cT$ of $\Omega$ consisting of closed simplices $T\in \cT$, we define the finite element space
\begin{equation}\label{Vtau}
\V_\cT \coloneqq H_0^1(\Omega) \cap \prod_{T \in \cT} \PP_n(T) \quad \an \V_k \coloneqq \V_{\cT_k},
\end{equation}
where $\PP_n(T)$ is the space of polynomials degree $\le n$ over $T$.
For any subset $\cS \subseteq \cT$,
\begin{equation}\label{Vtau_omega}
\V_\cT(\cS)  \coloneqq H_0^1(\Omega) \cap \hspace{-4pt} \prod_{T \in \cS} \hspace{-4pt} \PP_n(T).
\end{equation}

Given a triangulation $\cT$, we denote $h_{\cT}:= \max_{T\in \cT} h_{T}$ where $h_{T} := |T|^{1/d}.$ In particular, we denote $h_{0}:= h_{\cT_{0}}$ for an initial (conforming, shape-regular) triangulation $\cT_{0}$ of $\Omega$. Then the adaptive algorithm discussed below generates a nested sequence of conforming refinements $\{\cT_{k}\}$, with $\cT_{k}\ge \cT_{j}$ for $k\ge j$ meaning that  $\cT_{k}$ is  a conforming triangulation of $\Omega$ based on certain refinements of $\cT_{j}$. With this notation, we also simply denote by $\V_{k}:= \V_{\cT_{k}}$ the finite element space defined on $\cT_{k}$. 

The finite element approximation of the primal problem \eqref{primal_problem} reads: Find $u_k \in \V_k \st$
\begin{equation}\label{discrete_primal}
a(u_k, v_k) + \langle b(u_k), v_k \rangle= f(v_k), ~ v_k \in \V_k,
\end{equation}
and the finite element approximation of \eqref{Eapprox_dual} linearized about $u_j$ is given by: Find $\hat z^j_k \in \V_k $ such that
\begin{equation}\label{approx_dual_problem}
a(\hat z^j_k,v_k) + \langle b'(u_j) \hat z^j_k, v_k \rangle  = g(v_k) \quad \tforall v_k \in \V_k.
\end{equation}
Finally, for the purpose of analysis, we introduce the discrete limiting dual problem (cf. \eqref{limiting_dual_problem}) given by: Find $\hat z_k  \in \V_k $ such that
\begin{equation}\label{Ldiscrete_dual_problem}
a(\hat z_k,v_k) + \langle b'(u) \hat z_k, v_k \rangle  = g(v_k) \quad \tforall v_k \in \V_k.
\end{equation}

Existence and uniqueness of solutions to the primal problems~\eqref{primal_problem} and~\eqref{discrete_primal} follow from standard variational or fixed-point arguments as in~\cite{St00} and~\cite{Ke89}.  For the dual problems~\eqref{Eapprox_dual}-\eqref{limiting_dual_problem} and~\eqref{approx_dual_problem}-\eqref{Ldiscrete_dual_problem}  the existence and uniqueness of solutions follow from the standard Lax-Milgram Theorem as in~\cite{GT77}, since we assumed that $b'(\xi) \ge 0.$

We make the following assumption on the a priori $L_{\infty}$ bounds of the solutions to the primal problems~\eqref{primal_problem} and~\eqref{discrete_primal}:
\begin{assumption}[\emph{A priori bounds}]\label{a:apriori} Let $u$  and $u_{k}$ be the solution to~\eqref{primal_problem} and~\eqref{discrete_primal},respectively. We assume that there are $u_-, u_+ \in L_\infty$ which satisfy
\begin{equation}\label{discreteLinfBound}
u_-(x) \le u(x), u_k(x) \le u_+(x) ~\text{for almost every } x \in \Omega.
\end{equation}
\end{assumption}
\begin{remark}
\label{rk:Linfty}
The $L_\infty$ bound on $u$ follows from the standard maximum principle, as discussed in~\cite[Theorem 2.4]{BHSZ11a} and~\cite[Theorem 2.3]{Holst.M;Szypowski.R;Zhu.Y2013a}.  There is a significant literature on $L_\infty$ bounds for the discrete solution, usually requiring additional angle conditions on the triangulation (cf. \cite{Kerkhoven.T;Jerome.J1990,Jungel.A;Unterreiter.A2005,Karatson.J;Korotov.S2005,Holst.M;Szypowski.R;Zhu.Y2013a} and the references cited therein).  On the other hand, if $b$ satisfies the (sub)critical growth condition, as stated in~\cite[Assumption (A4)]{BHSZ11a}, then the $L_\infty$ bounds on the discrete solution $u_k$ are satisfied without angle conditions on the mesh; see~\cite{BHSZ11a} for more detail.
\end{remark}

Assumption~\ref{data_assumptions} together with Assumption~\ref{a:apriori} yield the following properties on the continuous and discrete solutions as summarized below.
\begin{proposition}\label{assumptionsonb} Let the problem data satisfy Assumption~\ref{data_assumptions} and Assumption~\ref{a:apriori}. The following properties hold:
\begin{enumerateX}
\item $b$ is Lipschitz  on $[u_-,u_+]\cap H_0^1(\Omega)$ for a.e. $x \in \Omega$ with constant $B$.
\item $b'$ is Lipschitz on $[u_-,u_+]\cap H_0^1(\Omega)$ for a.e. $x \in \Omega$ with constant $\Theta$. 
\item Let $\hat z$  bet the solution to~\eqref{limiting_dual_problem}, $\hat z_j^j$  the solution to~\eqref{approx_dual_problem}  and $\hat z_j$  the solution to~\eqref{Ldiscrete_dual_problem}. Then there are $z_-, z_+ \in L_\infty$ which satisfy
\begin{equation}\label{dualLinfBound}
z_-(x) < \hat z(x), \hat z_j(x), \hat z_j^j(x)  \le z_+(x) ~\text{for almost every } x \in \Omega, ~ j \in \N.
\end{equation}
\end{enumerateX}
\end{proposition}

\section{Goal Oriented AFEM}   \label{sec:goafem}
In this section, we describe the goal oriented adaptive finite element method (GOAFEM), which is based on the 
standard AFEM algorithm:
\begin{equation}\label{goafem00}
\text{ SOLVE } \rightarrow \text{ ESTIMATE } \rightarrow \text{ MARK } \rightarrow \text{ REFINE }.
\end{equation} 
Below, we explain each procedure. 

{\noindent\bf\em Procedure SOLVE.}   The procedure SOLVE involves solving~\eqref{discrete_primal} for  $u_j$, computing  $b'(u_j)$  to form problem \eqref{approx_dual_problem} and solving~\eqref{approx_dual_problem} for  $\hat z_j^j$. In the analysis that follows, we assume for simplicity that the exact Galerkin solution is found on each mesh refinement. In practice the nonlinear problem~\eqref{discrete_primal} may be solved by a standard inexact Newton + multilevel algorithm as in~\cite{BHSZ11b}. The approximate dual problem~\eqref{approx_dual_problem} may be solved by any standard linear-time iterative method.

{\noindent\bf\em Procedure ESTIMATE.}
We use a standard residual-based element-wise error estimator for both primal and approximate dual problems. Recall that the residual of the primal problem is given by $R(v) = f - \cN(v)$ with $\cN(v) = -\nabla \cdot (A \nabla v) + b(v)$.
For the limiting and approximate dual problems, we define the local strong form by
$
\hat \cL^\ast (v) \coloneqq -\nabla \cdot (A \nabla v) + b'(u)(v), ~\an~
\hat \cL_j^\ast (v) \coloneqq -\nabla \cdot (A \nabla v) + b'(u_j)(v).
$ 
The limiting and approximate dual residuals given respectively by
\begin{equation}\label{dualresi} 
R^\ast(v) \coloneqq g - \hat \cL^\ast(v), ~\an~ \hat R_j^\ast(v) \coloneqq  g - \hat \cL_j^\ast(v).
\end{equation}
The {\em jump residual} for both the primal and linearized dual problems is:
$$
J_T(v) \coloneqq \lbb [A \nabla v] \cdot n \rbb_{\pa T},
$$
where $\lbb ~\cdot~ \rbb$ is given by
$
\lbb \phi \rbb_{\pa T} \coloneqq \lim_{t \goto 0} \phi(x + t n) - \phi(x - tn)
$
 and $n$ is taken to be the appropriate outward normal defined on $\pa T$.
The error indicator for the primal problem \eqref{discrete_primal} is given by
\begin{equation}\label{eta_cmpct}
\eta_\cT^2(v,T) \coloneqq h_T^2 \nr{R(v)}_{L_2(T)}^2 +  h_T \nr{ J_T(v)  }_{L_2(\pa T)}^2,   \quad v \in \V_\cT.
\end{equation}
Similarly, the dual error-indicator is given by the approximate residual
\begin{equation}\label{zeta_cmpct}
\zeta_{\cT,j}^2(w,T) \coloneqq h_T^2 \nr{\hat R_j^\ast( w)}_{L_2(T)}^2 + 
h_T \nr{  J_T( w)   }_{L_2(\pa T)}^2, \quad w \in \V_\cT.
\end{equation}
This dual indicator is defined in terms of the approximate dual operator $b'(u_j)$ as this is a computable quantity given an approximation $u_j$. In addition, for purpose of analysis we define the limiting dual error-indicator by
\begin{equation}\label{limitingInd}
\zeta_{\cT}^2(w,T) \coloneqq h_T^2 \nr{\hat R^\ast( w)}_{L_2(T)}^2 + 
h_T \nr{  J_T( w)   }_{L_2(\pa T)}^2, \quad w \in \V_\cT.
\end{equation}
We remark that the limiting dual indicator as given by~\eqref{limitingInd} is not computable. For any given subset $\cS\subset \cT$, the error estimators on $\cS$ are given by the $l_2$ sum of error indicators over elements in the space.
$$
\eta_\cT^2(v, \cS)   \coloneqq    \sum_{T \in \cS} \eta_{\cT}^2(v,T) , \quad v \in \V_\cT.
$$
The dual energy estimator is:
$$
\zeta_{\cT,j}^2(w,\cS) \coloneqq  \sum_{T \in \cS} \zeta_{\cT,j}^2(w,T), \quad w \in \V_\cT,
$$
and the limiting estimator
$$
\zeta_{\cT}^2(w,\cS) \coloneqq  \sum_{T \in \cS} \zeta_{\cT}^2(w,T), \quad w \in \V_\cT.
$$
To simplify the notation, below we will omit ``$\cS$'' in the above definitions if $\cS = \cT$ and we will use $\eta_{k}$ to denote $\eta_{\cT_{k}}$, and similarly use $\zeta_{k,\cdot}$ to denote $\zeta_{\cT_{k}, \cdot}$.

As in~\cite{CKNS08} it is not difficult to verify that the indicators for the primal and approximate (respectively limiting) dual problems satisfy the monotonicity property for $v \in \V(\cT_1)$ and  $\cT_2 \ge \cT_1$
\begin{equation}\label{indicator_mono}
\eta_2(v,\cT_2) \le \eta_1(v, \cT_1), ~\zeta_{2,j}(v,\cT_2) \le \zeta_{1,j}(v,\cT_1) 
~\an~ \zeta_{2}(v,\cT_2) \le \zeta_{1}(v,\cT_1)  .
\end{equation}
For an element $T \in \cT_2 \cap \cT_1$
\begin{equation}\label{indicator_local}
\eta_2(v,T) = \eta_1(v, T),  ~\zeta_{2,j}(v,T) = \zeta_{1,j}(v,T)
~\an~ \zeta_{2}(v,T) = \zeta_{1}(v,T).
\end{equation}

{\noindent\bf\em Procedure MARK.}
The D\"orfler  marking strategy for the goal-oriented problem is based on the following steps as in~\cite{MoSt09}:  
\begin{enumerateX}
\item
Given $\theta \in (0,1)$, mark sets for each of the primal and dual problems:
\begin{itemizeX}
\item Mark a set $\cM_p \subset \cT_k$ such that
\begin{equation}\label{markP}
 \eta_k^2(u_k,\cM_{p}) \ge \theta^2 \eta_k^2(u_k, \cT_k).
\end{equation}
\item Mark a set $\cM_d \subset \cT_k$ such that
\begin{equation}\label{markD}
 \zeta_{k,k}^2(\hat z_k^k,\cM_{d}) \ge \theta^2 \zeta_{k,k}^2(\hat z_k^k, \cT_k).
\end{equation}
\end{itemizeX}
\item
Let $ \cM = \cM_p \cup \cM_d$  be  the union of  sets found for the primal and dual problems respectively.  
\end{enumerateX}
As in ~\cite{HoPo11a} the set $\cM$ differs from that in~\cite{MoSt09}, where the set of lesser cardinality between $\cM_p \an \cM_d$ is used.  We emphasize the necessity of this choice to obtain strong contraction both in terms of the primal problem and the combined primal-dual system.  For Poisson's equation investigated in~\cite{MoSt09}, the contracting quantity is the energy error, whereas here we develop contraction arguments for the quasi-error which combines the energy error with the error estimator.  As the sequence of estimators for the primal (dual) problem based on the latest solution at each iteration is not necessarily monotone decreasing unless the primal (dual) problem has been refined for, we refine for both primal and dual problems at each iteration in order to force convergence of the quasi-error for both primal and dual problems.   As seen in~\eqref{markD} the mesh is marked with respect to the dual indicators of the approximate-sequence solutions $\hat z_k^k$ as these are computable quantities.   
Sets $\cM_p \an \cM_d$ with optimal cardinality (up to a factor of 2) can be chosen in linear time~ by binning the elements rather than performing a full sort~\cite{MoSt09}.

{\noindent\bf\em Procedure REFINE.}
The refinement (including the completion) is performed according to newest vertex bisection which was first proposed in \cite{Sewell.E1972}. It has been proved that the bisection procedure will preserve the shape-regularity of the initial triangulation $\cT_{0}$.  The complexity and other properties of this procedure are now well-understood (see for example~\cite{BDD04} and the references cited therein), and will simply be exploited here.

\section{Contraction for the primal problem}
   \label{sec:primalC}

In this section, we discuss the contraction of the primal problem~\eqref{primal_problem}, recalling results from~\cite{HTZ09a}, \cite{Holst.M;Szypowski.R;Zhu.Y2013a} and~\cite{BHSZ11b}. The contraction argument relies on  three main convergence results, namely quasi-orthogonality, error-estimator as upper bound on error and estimator reduction. We include the analogous results here for the limiting dual problem when they are identical or nearly identical.
\subsection{Quasi-orthogonality }\label{subsec:quasi_orthog}

Orthogonality in the energy-norm $\nrse{u - u_{2}}^2= \nrse{u - u_1}^2 - \nrse{u_{2} - u_1}^2$ does not generally hold in the semilinear problem.
We rely on the weaker quasi-orthogonality result to establish contraction of AFEM (GOAFEM).
The proof of the quasi-orthogonality relies on the following $L_{2}$-lifting property. 
\begin{lemma}[$L_2$-lifting]\label{L2lifting}
Let the problem data satisfy Assumption~\ref{data_assumptions} and Assumption~\ref{a:apriori}. Let $u$ be the exact solution to~\eqref{primal_problem}, and $u_1 \in \V_1$ the Galerkin solution to~\eqref{discrete_primal}. Let $w\in H^{1+s}(\Omega) \cap H_0^1(\Omega)$ for some $0 < s \le 1$ be the solution to the dual problem: Find $w\in H_{0}^{1}(\Omega)$ such that
\begin{equation}\label{duality_dual}
a(w,v) + \langle \cB_1 w, v \rangle= \langle  u - u_1, v \rangle, \quad v \in H_0^1(\Omega),
\end{equation}
where the operator $\cB_{1}$ is defined by $\cB_1 \coloneqq  \int_0^1 b'(\xi u + (1 - \xi) u_1 ) ~d \xi$.
As in \cite{Ci02, Evans98, AOB01} we assume the regularity
\begin{equation}\label{elliptic_reg0A}
|w|_{H^{1+s}(\Omega)} \le K_R \nr{u-u_{1}}_{L_2(\Omega)}
\end{equation}
based on the continuity of the coefficients $a_{ij}$ and of $b'(\cdot)$ .
Then
\begin{align}\label{duality_res0A}
\nr{u - u_1}_{L_2} & \le C_\ast h_0^s \nrse{u - u_1}.
\end{align}
\end{lemma}
\begin{proof} The proof follows the standard duality arguments in~\cite{AOB01}, \cite{HoPo11a} and~\cite{BS08}, adapted for the semilinear problem. Let $\cI^h: H_{0}^{1}(\Omega) \to \V_{1}$ be a quasi-interpolator, satisfying
\begin{align}\label{interpolation_estA}
\nr{w - \cI^h w}_{H^1} &\le C_\cI h_{\cT_1}^s |w|_{H^{1+s}} \\
\label{interpolation_estB}
\nr{w - \cI^h w}_{L_2} &\le \hat C_\cI h_{\cT_1}^{1+s} |w|_{H^{1+s}}.
\end{align}
as discussed in~\cite{AOB01}, \cite{SF73} and~\cite{HoPo11a}.

Consider the linearized dual problem~\eqref{duality_dual} with $v = u - u_1 \in H_0^1(\Omega)$
expressed in primal form
\begin{equation}\label{lift2}
a(u- u_1,w) + \langle \cB_1(u - u_1), w \rangle = \nr{u - u_1}_{L_2}^2.
\end{equation}
By Galerkin orthogonality, for $\cI^hw \in \V_1$
\begin{equation}\label{lift3}
a(u - u_1, \cI^hw) + \langle \cB_1(u - u_1), \cI^h w\rangle = 0.
\end{equation}
Subtracting~\eqref{lift3} from~\eqref{lift2}
\begin{equation}
a(u- u_1,w- \cI^hw) + \langle b(u) -b(u_1), w - \cI^hw\rangle = \nr{u - u_1}_{L_2}^2.
\end{equation}
Then by~\eqref{continuity} continuity of $\forma$, the H\"older inequality and Lipschitz continuity of $b$ (Proposition \ref{assumptionsonb}):
\begin{align}\label{lift4}
\nr{u - u_1}_{L_2}^2 & \le M_\cE\nr{u - u_1}_{H^1}\nr{w - \cI^hw}_{H^1} + B \nr{u - u_1}_{L_2}\nr{w - \cI^hw}_{L_2}.
\end{align}
By coercivity~\eqref{coercive}, interpolation estimate~\eqref{interpolation_estA}, and regularity~\eqref{elliptic_reg0A} on the first term on the RHS of~\eqref{lift4}
\begin{align}\label{lift5}
M_\cE\nr{u - u_1}_{H^1}\nr{w - \cI^hw}_{H^1} & \le \f{M_\cE}{m_\cE} C_\cI h_0^s \nrse{u - u_1}|w|_{H^{1+s}}
\nonumber \\
& \le  \f{M_\cE}{m_\cE} K_R C_\cI h_0^s \nrse{u - u_1}\nr{u - u_1}_{L_2}.
\end{align}
For the second term of~\eqref{lift4}, apply~\eqref{interpolation_estB} followed by~\eqref{elliptic_reg0A} and coercivity to the interpolation error yielding
\begin{align}\label{lift6}
B \nr{u - u_1}_{L_2}\nr{w - \cI^hw}_{L_2} & \le B \hat C_\cI h_0^{1+s} \nr{u - u_1}_{L_2}|w|_{H^{1+s}}
\nonumber \\
& \le   K_R B\hat C_\cI  h_0^{1+s} \nr{u - u_1}_{L_2}\nr{u - u_1}_{L_2} 
\nonumber \\
& \le  (m_\cE^{-1} K_R B\hat C_\cI h_0)h_0^s \nr{u - u_1}_{L_2}\nrse{u - u_1}.
\end{align}
Applying~\eqref{lift5} and~\eqref{lift6} to~\eqref{lift4}, we obtain
\begin{align}\label{lift7}
\nr{u - u_1}_{L_2} & \le m_\cE^{-1}K_R \left(M_\cE  C_\cI  +  B\hat C_\cI h_0 \right)h_0^s \nrse{u - u_1}.
\end{align}
This completes the proof.
\end{proof}

Similarly, we have the following $L_{2}$-lifting result for two Galerkin solutions.
\begin{corollary}\label{TheOtherL2Lift}
Let the assumptions in Lemma~\ref{L2lifting} hold. Let $u_1 \in \V_1$ and $u_2 \in \V_2$ be the Galerkin solutions to~\eqref{discrete_primal} in the spaces $\V_{1} \subset \V_{2}$, respectively. 
Then there is a constant $C_{\ast}>0$ such that 
\begin{equation}\label{TCL2}
\nr{u_2 - u_1}_{L_2} \le C_\ast h_0^s \nrse{u_2 - u_1}.
\end{equation}
\end{corollary}
\begin{proof}
The proof of~\eqref{TCL2} follows by replacing $u$ by $u_2$ in Lemma~\ref{L2lifting}.  In this case, we should replace the dual problem~\eqref{duality_dual} by: Find $w \in \V_2$ such that
\begin{equation}\label{duality_dual_second}
a(w,v) + \langle \cB_{12} w, v \rangle= ( u_2 - u_1, v ), \quad v \in \V_2,
\end{equation}
where the operator $\cB_{12} \coloneqq  \int_0^1 b'(\xi u_2 + (1 - \xi) u_1 ) ~d \xi$.
The rest of the proof is the same as Lemma~\ref{L2lifting}.
\end{proof}
\begin{remark}
As the dual problem~\eqref{duality_dual} changes at each iteration, so may the regularity constant as given by~\eqref{elliptic_reg0A} as well as the interpolation constants as given by~\eqref{interpolation_estA} and~\eqref{interpolation_estB}. As such, the previous lemma shows a $C_{\ast,k}$ for $k = 1, 2, \ldots$.  As the algorithm is run finitely many times, we consolidate these $C_{\ast,k}$ into a single constant $C_\ast$ for simplicity of presentation.
\end{remark}

Now we are in position to show the quasi-orthogonality.
\begin{lemma}[Quasi-orthogonality]\label{primal_quasi}
Let the problem data satisfy Assumption~\ref{data_assumptions} and Assumption~\ref{a:apriori}. Let $\cT_1, \cT_2$ be two conforming triangulation of $\Omega$ with $\cT_2 \ge \cT_1$. Let $u\in H_0^1(\Omega)$ be the exact solution to~\eqref{primal_problem}, $u_i \in \V_i$ the solution to~\eqref{discrete_primal}, $i = 1,2$.
There exists a constant $C_\ast > 0$ depending on the problem data $\bD$ and initial mesh $\cT_0$, and a number $0 < s \le 1$ related to the angles of $\pa \Omega$, such that if the meshsize $h_0$ of the initial mesh satisfies $\bar \Lambda \coloneqq  Bm_\cE^{-1}C_\ast  h_0^{s} < 1$, then
\begin{equation}\label{quasi_ortho}
\nrse{u - u_{2}}^2 \le \Lambda \nrse{u - \bar v}^2 - \nrse{u_{2} - \bar v}^2, \quad \forall \bar v \in \V_{2},
\end{equation}
and in particular for $\bar v = u_1\in \V_{1}\subset \V_{2}$
\begin{equation}\label{galerkin_qo}
\nrse{u - u_{2}}^2 \le \Lambda_G \nrse{u - u_1}^2 - \nrse{u_{2} - u_1}^2,
\end{equation}
where
\[
\Lambda \coloneqq (1 - B m_\cE^{-1}C_\ast  h_0^{s})^{-1} ~\an~ \Lambda_G \coloneqq (1 - BC_\ast^2 h_0^{2s})^{-1} 
\]
and $C_\ast$ is the constant from Lemma~\ref{L2lifting}.
\end{lemma}
\begin{proof}
For any given $\bar v \in \V_{2}$, we have 
\begin{align} \label{QP1} 
\nrse{u - u_2}^2 
& = \nrse{u - \bar v}^2 - \nrse{\bar v - u_2}^2 + 2a(u - u_2, \bar v - u_2). 
\end{align}
By Galerkin orthogonality
\begin{equation}\label{QP2}
a(u - u_2, v) + \langle b(u) - b(u_2), v \rangle = 0 \tforall v \in \V_2,
\end{equation}
and taking $v = \bar v - u_2$ in~\eqref{QP2}, we have 
\begin{align}\label{QP3}
2a(u - u_2, \bar v - u_2) & \le 2 |\langle b(u) - b(u_2), \bar v - u_2\rangle| \nonumber \\
& \le 2B \nr{u - u_2}_{L_2} \nr{\bar v - u_2}_{L_2}.
\end{align}
Here we used H\"older inequality and the Lipschitz property on $b$ (cf. Proposition~\ref{assumptionsonb}).

To prove the inequality~\eqref{quasi_ortho}, by applying the $L_2$-lifting Lemma~\ref{L2lifting}  to the first factor on the RHS and the coercivity~\eqref{coercive} to the second followed by Young's inequality, we obtain
\begin{align}\label{QP4a}
 2B \nr{u - u_2}_{L_2} \nr{\bar v - u_2}_{L_2} & \le 2B m_\cE^{-1}C_\ast  h_0^{s}\nrse{u - u_2}\nrse{\bar v - u_2} \nonumber \\
 & \le B m_\cE^{-1}C_\ast  h_0^{s}\nrse{u - u_2}^2 + B m_\cE^{-1}C_\ast h_0^{s}\nrse{\bar v - u_2}^2.
\end{align}
Applying~\eqref{QP4a} via~\eqref{QP3} to~\eqref{QP1}
\[
(1 - B m_\cE^{-1}C_\ast  h_0^{s})\nrse{u - u_2}^2 \le \nrse{u - \bar v}^2 - (1 - B m_\cE^{-1}C_\ast  h_0^{s})\nrse{\bar v - u_2}^2.
\]
Assuming 
$
\bar \Lambda \coloneqq  Bm_\cE^{-1}C_\ast  h_0^{s} < 1,
$ 
we have
\begin{equation}
\nrse{u - u_2}^2 \le \Lambda\nrse{u - \bar v}^2 - \nrse{\bar v - u_2}^2
\end{equation}
with $\Lambda = (1 - Bm_\cE^{-1}C_\ast  h_0^{s})^{-1}$.

The proof of the inequality~\eqref{galerkin_qo} is almost identical. By applying $L_2$-lifting~\ref{L2lifting} to each norm on the RHS of~\eqref{QP3} by means of Corollary~\ref{TheOtherL2Lift} then applying Young's inequality
\begin{align}\label{QP4}
 2B \nr{u - u_2}_{L_2} \nr{u_1 - u_2}_{L_2} & \le 2B h_0^{2s}C_\ast^2 \nrse{u - u_2}\nrse{u_1 - u_2} \nonumber \\
 & \le B h_0^{2s}C_\ast^2 \nrse{u - u_2}^2 + BC_\ast^2 h_0^{2s} \nrse{u_1 - u_2}^2.
\end{align}
Following the same procedure as above yields
\begin{equation}\label{QP4G}
\nrse{u - u_2}^2 \le \Lambda_G\nrse{u - u_1}^2 - \nrse{u_1- u_2}^2
\end{equation}
with $\Lambda_G = (1 - BC_\ast^2 h_0^{2s})^{-1}$ with the weaker mesh assumption $\bar \Lambda_G \coloneqq BC_\ast^2h_0^{2s}<1$.
\end{proof}

We note that the second Galerkin orthogonality estimate~\eqref{QP4G} sharpens our results but is not essential to establishing them.
\subsection{Error Estimator as Global Upper-bound}\label{subsec:upper}
The second key result for the contraction of the primal problem is the error estimator as a global upper bound on the energy error, up to a global constant.  The result for the semilinear problem is established in \cite{HTZ09a,Holst.M;McCammon.J;Yu.Z;Zhou.Y2012} with a clear generalization to the approximate dual sequence, also see~\cite{CKNS08} and~\cite{MeNo05} for the linear cases. The proof of this result follows from the general \emph{a posteriori} error estimation framework developed in \cite{Verfurth.R1994,Verfurth.R1996}.
\begin{lemma}[Error estimator as global upper-bound]\label{primal_upper}
Let the problem data satisfy Assumption~\ref{data_assumptions} and Assumption~\ref{a:apriori}. Let $\cT_{k}$ be a conforming refinement of $\cT_{0}$.  Let $u\in H_{0}^{1}(\Omega)$ and $u_k \in \V_k$ be the solutions to~\eqref{primal_problem} and~\eqref{discrete_primal}, respectively.  Similarly, let  $\hat z \in H_{0}^{1}(\Omega)$ and $\hat z_k \in \V_k$ be the solutions to~\eqref{limiting_dual_problem} and~\eqref{Ldiscrete_dual_problem}, respectively.
Then there is a global constant $C_1$ depending only on the problem data $\bD$ and initial mesh $\cT_0$ such that
\begin{equation}\label{estimator_ge_error}
\nrse{u - u_k} \le C_1 \eta_k(u_k)
\end{equation}
and
\begin{equation}\label{dual_upper}
\nrse{\hat z - \hat z_k}  \le C_1 \zeta_{k}(\hat z_k).
\end{equation}
\end{lemma}

\subsection{Estimator Reduction}\label{subsec:estim_reduc}
The local Lipschitz property as in~\cite{HTZ09a}, analogous to the local perturbation property established in~\cite{CKNS08}, is a key step in establishing  estimator reduction leading to the contraction  result.  For any $T\in \cT$,  we denote
\begin{equation}\label{def:patch}
\omega_T \coloneqq T \cup \{ T' \in \cT ~\rest~ T \cap T' \text{ is a true-hyperface of } T \}.
\end{equation}
Here, for a $d$-simplex $T$, a true-hyperface is a $d-1$ sub-simplex of $T$, e.g., a face in 3D or an edge in 2D. We also define the data estimator on each element $T \in \cT$ as 
\begin{equation}\label{data_est_element_NL}
\eta_\cT^2(\bD,T) = h_T^2 \left( \nr{\div A}_{L_\infty(T)}^2 +  h_T^{-2 } \nr {A}_{L_\infty(\omega_T)}^2 + B^2 \right),
\end{equation}
 and denote  
$
\eta_\cT (\bD,\cS) = \max_{T \in \cS} \eta_{\cT}(\bD,T)
$
for any subset $\cS\subseteq \cT$. Recall that $B$ is the Lipschitz constant in Proposition~~\ref{assumptionsonb}. In particular, we denote by $
\eta_0 \coloneqq \eta_{\cT_0} (\bD,\cT_0) $ the data estimator on the initial mesh.  
As the grid is refined, the data estimator  satisfies the monotonicity property for refinements $\cT_2 \ge \cT_1$ (cf.~\cite{CKNS08}): 
\begin{align}\label{mono_data}
\eta_2(\bD,\cT_2) \le \eta_1(\bD,\cT_1).
\end{align}
\begin{lemma}[Local Lipschitz Property]\label{primalLL}
Let the problem data satisfy Assumption~\ref{data_assumptions} and Assumption~\ref{a:apriori}. Let $\cT$ be a conforming refinement of $\cT_{0}$. Then for all $T \in \cT$ and for any  $v, w \in \V_\cT$
\begin{align}\label{perturb_est}
|\eta_\cT(v,T) -  \eta_\cT(w,T)| \le \bar \Lambda_1 \eta_\cT(\bD,T)\nr{v - w}_{H^1(\omega_T)}.
\end{align}
 The constant $\bar \Lambda_1> 0$ depends on the dimension $d$ and the initial mesh $\cT_0$.
\end{lemma}
\begin{proof}
The proof follows those in~\cite{CKNS08} and \cite{HoPo11a}, and we sketch the proof below.
From~\eqref{eta_cmpct}
\begin{equation}
\eta_\cT^2(v,T) \coloneqq h_T^2 \nr{R(v)}_{L_2(T)}^2 +  h_T \nr{ J_T(v)    }_{L_2(\pa T)}^2,   \quad v \in \V_\cT.
\end{equation}
Set $e = v - w$ and by definition of the residual, we get
\begin{align*}
R(v) & = f - \cN( w + e) \\
& = f + \nabla \cdot (A \nabla w) -  b(w) +\nabla \cdot (A \nabla e)  - \left(\int_0^1 b'(w + \xi e) ~d\xi \right)e \\
& = R(w) + \cD(e),
\end{align*}
where $\cD(e) \coloneqq  \nabla \cdot (A \nabla e) - \left(\int_0^1 b'(w + \xi e) ~d\xi\right) e$.  
Using the generalized triangle-inequality
\begin{equation*}\label{subsec:triangle}
\sqrt{(a + b)^2 + (c + d)^2} \le \sqrt{a^2 + c^2} + b + d, \quad \text{ for }  a,b,c,d > 0
\end{equation*}
and linearity of the jump residual we have
\begin{align}\label{L0}
\eta_{\cT}(v,T) & = \left( h_T^2 \nr{R(w) + \cD(e)}_{L_2(T)}^2 + h_T \nr{J(w) + J(e)}_{L_2(\pa T)}^2\right)^{1/2} 
\nonumber \\
& \le \eta_{\cT}(w,T) + h_T\nr{\cD (e)}_{L_2(T)} + h_T^\half \nr{J(e)}_{L_2(\pa T)} .
\end{align}
For the second term of \eqref{L0}, by triangle inequality we obtain
\begin{equation}\label{L1}
\nr{\cD(e)}_{L_2(T)}  \le \nr{\nabla \cdot (A \nabla e)}_{L_2(T)} + \left\| {\left(\int_0^1 b'(w + \xi e) ~d\xi\right) e } \right\|_{L_2(T)}.
\end{equation}
By the inverse inequality, the diffusion term satisfies the bound
 \begin{align}\label{L2}
\nr{\nabla \cdot (A \nabla e)}_{L_2(T)} 
& \le \nr{\div A \cdot \nabla e}_{L_2(T)} + \nr{A : D^2 e}_{L_2(T)}  \nonumber \\ 
 & \le  \left( \nr{\div A}_{L_\infty(T)} +  C_I h_T^{-1 } \nr {A}_{L_\infty(T)}  \right) \nr{\nabla e}_{L_2(T)},
\end{align}
where $D^{2} e$ is the Hessian of $e$. The second term in ~\eqref{L1} is bounded by
\begin{equation}\label{L3}
\left\| { \left(\int_0^1 b'(w + \xi e) ~d\xi \right)e } \right\|_{L_2(T)} \le B \nr{e}_{L_2(T)}.
\end{equation}
The jump term in~\eqref{L0} satisfies
\begin{align}\label{L4}
\nr{J(e)}_{L_2(\pa T)} & \le 2 (d+1)~ C_T  h_{T}^{-\half}  \nr A_{L_\infty(\omega_T)} \nr{\nabla e}_{L_2(\omega_T)}
\nonumber \\
& = C_Jh_{T}^{-\half} \nr A_{L_\infty(\omega_T)} \nr{\nabla e}_{L_2(\omega_T)},
\end{align}
where $C_T$ depends only on the shape-regularity of the triangulation. 
Putting together~\eqref{L0}, \eqref{L2}, \eqref{L3} and~\eqref{L4}, we obtain
\begin{align}\label{L5}
\eta_{\cT}(v,T) & \le \eta_{\cT}(w,T) + h_T \left( \nr{\div A}_{L_\infty(T)} +  (C_I + C_J) h_T^{-1 } \nr {A}_{L_\infty(\omega_T)} + B \right)  \nr{ e}_{H^1(\omega_T)} 
\nonumber \\
& \le  \eta_{\cT}(w,T) + C_{TOT}\, \eta_\cT(\bD,T)\nr{v - w}_{H^1(\omega_T)}.
\end{align}
This completes the proof.
\end{proof}

The local perturbation property as demonstrated in Lemma~\ref{primalLL} (respectively, Lemma~\ref{dualLL} below) leads to estimator reduction, one of the three key ingredients for contraction of the both the primal and combined quasi-errors.  This result holds for both the primal and limiting dual problems, whose proof can be found in~\cite[Corollary 2.4]{CKNS08} or~\cite[Theorem 3.4]{HoPo11a}. 
\begin{theorem}[Estimator reduction]\label{lemma:est_reduc}
Let the problem data satisfy Assumption~\ref{data_assumptions} and Assumption~\ref{a:apriori}. Let $\cT_1$ be a conforming refinements of $\cT_{0}$, $\cM \subset \cT_1$ be the marked set, and $\cT_2 = \text{REFINE}(\cT_1, \cM)$. Let
\[
\Lambda_1 \coloneqq (d+2)\bar \Lambda_1^2 m_\cE^{-2} \quad \an~  \lambda \coloneqq 1 - 2^{-1/d}> 0
\]
with $\bar \Lambda_1$ from Lemma~\ref{primalLL} (local Lipschitz property). Then  for any $v_1 \in \V_1$ and $v_2 \in \V_2$ and $\delta > 0$ 
\begin{align}\label{est_reduc}
\eta_2^2(v_2, \cT_2)  \le  &(1 + \delta) \left\{ \eta_1^2(v_1,\cT_1) - \lambda \eta_1^2(v_1,\cM) \right\}  + ( 1 + \delta^{-1}) \Lambda_1 \eta_0^2 \nrse{v_2 - v_1}^2.
\end{align}
Analogously for the limiting dual problem
\begin{align}\label{Dest_reduc}
\zeta_{2}^2(v_2, \cT_2)  \le  &(1 + \delta) \left\{ \zeta_{1}^2(v_1,\cT_1) - \lambda \zeta_{1}^2(v_1,\cM) \right\}  + ( 1 + \delta^{-1}) \Lambda_1 \eta_0^2 \nrse{v_2 - v_1}^2.
\end{align}
\end{theorem}
The contraction of the primal (semilinear) problem is established in~\cite{HTZ09a} and~\cite{Holst.M;McCammon.J;Yu.Z;Zhou.Y2012} based on Lemma~\ref{primal_quasi}, Lemma~\ref{primal_upper} and Theorem~\ref{lemma:est_reduc} as discussed above. 
\begin{theorem}[Contraction of the primal problem]\label{primal_contraction}
Let the problem data satisfy Assumption~\ref{data_assumptions} and Assumption~\ref{a:apriori}. Let $u$ the solution to~\eqref{primal_problem}. 
Let $\theta \in (0,1]$, and let $\{\cT_j, \V_j, u_j\}_{j \ge 0}$ be the sequence of meshes, finite element spaces and discrete solutions produced by GOAFEM.  Then there exist constants $\gamma_p > 0 \an 0 < \alpha < 1$, depending on the initial mesh $\cT_0$ and marking parameter $\theta$ such that
\begin{equation}\label{contractionres}
\nrse{u - u_{j+1}}^2 + \gamma_p \eta_{j+1}^2 \le \alpha^2\left(  \nrse{u - u_j}^2 + \gamma_p \eta_j^2 \right).
\end{equation}
\end{theorem}

\section{Contraction and Convergence of GOAFEM}
   \label{sec:contraction_thms}

In this section, we discuss the contraction and convergence of the GOAFEM described in \S\ref{sec:goafem}. In particular, we show that the GOAFEM algorithm generates a sequence $\{\cT_j, \V_j, u_j, \hat z_{j}\}_{j \ge 0}$ which contracts not only in the primal error as shown in \S\ref{sec:primalC}, but also in a linear combination of the primal and limiting dual error. We emphasize that it would be difficult to derive convergence results in terms of problem \eqref{linearized_dual} or \eqref{Eapprox_dual}, because at each refinement the problem is changing. So we show contraction in terms of the error in the limiting dual problem \eqref{limiting_dual_problem} as the target equation is fixed over the entire adaptive algorithm. Our approach of showing contraction in this section again relies on three main components: quasi-orthogonality, error-estimator as upper bound on error and estimator reduction.  Here we discuss the relevant results for the limiting dual problem with an emphasis on those that differ significantly from the corresponding  results for the primal problem. 
Note the limiting dual problem is not computable. We connect the error for the limiting dual problem to the computable quantities in the GOAFEM algorithm. For this purpose, we introduce Lemma~\ref{estimatorSwitch}, converting between limiting and approximate estimators in order to apply the D\"orfler property to a computable quantity; and Lemma~\ref{DPErr}, bounding the discrete error between approximate and limiting dual solutions in terms of the primal error.  We put these results together in Theorem~\ref{quasi_combine_C} to establish the contraction of the combined quasi-error.  Finally, the contraction of this form of the error is related to the error in the quantity of interest in Corollary~\ref{goalErrorBound}.

\subsection{Quasi-orthogonality for Limiting-dual Problem}
\label{subsec:Dquasi_orthog}
Similar to the proof of the quasi-orthogonality for the primal problem,  we make use of an $L_{2}$-lifting argument for the limiting-dual problem. Let $\hat z \in H_0^1(\Omega)$  and $\hat z_1 \in \V_1$ be the solutions to ~\eqref{limiting_dual_problem} and~\eqref{Ldiscrete_dual_problem}, respectively. We again use the duality argument, and introduce the problem: Find $y \in H_0^1(\Omega)$ such that
\begin{equation}\label{linearized_adjoint}
a(y,v) + \langle b'(u) y, v\rangle = (\hat z - \hat z_{1}, v) \tforall v \in H_0^1(\Omega)
\end{equation}
Then we have the following $L_{2}$-lifting result for the limiting-dual problem.
\begin{lemma}[Limiting-dual $L_2$-lifting]\label{dualL2lifting}
Let the problem data satisfy Assumption~\ref{data_assumptions} and Assumption~\ref{a:apriori}. Let $\cT_1$ be a conforming triangulation, and $\hat z \in H_0^1(\Omega)$  and $\hat z_1 \in \V_1$ be the solutions to~\eqref{limiting_dual_problem} and~\eqref{Ldiscrete_dual_problem}, respectively. 
Assume that the solution $y$ to \eqref{linearized_adjoint} belongs to $H^{1+s}(\Omega) \cap H_0^1(\Omega)$ for some $0 < s \le 1$ such that 
\begin{equation}\label{Delliptic_reg0A}
|y|_{H^{1+s}(\Omega)} \le \bar  K_R \nr {\hat z - \hat z_1}_{L_2(\Omega)}.
\end{equation}
Then
\begin{align}\label{Dduality_res0A}
\nr{\hat z - \hat z_1}_{L_{2}} & \le \hat C_\ast h_0^s \nrse{\hat z - \hat z_1}.
\end{align}
\end{lemma}
\begin{proof}
The proof is essentially the same as that of Lemma \ref{L2lifting}, we omit here. 
\end{proof}
\begin{remark}
Similar to Corollary~\ref{TheOtherL2Lift}, the $L_2$-lifting Lemma~\ref{dualL2lifting} also holds for two Galerkin solutions to \eqref{Ldiscrete_dual_problem}, $\hat z_2 \in \V_2 \an \hat z_1\in V_1$ with $\V_{1} \subset \V_{2}$, namely, 
\begin{equation*}
\nr{\hat z_{2} - \hat z_1}_{L_{2}}  \le \hat C_\ast h_0^s \nrse{\hat z_{2} - \hat z_1}.
\end{equation*} 
The proof is essentially the same. We only need to replace ~\eqref{linearized_adjoint} by the problem: Find $y \in \V_2$ such that
\begin{equation*}
a(y,v) + \langle b'(u) y, v\rangle = \langle \hat z_2 - \hat z_{1}, v\rangle \tforall v \in \V_2.
\end{equation*} 
\end{remark}

With the help of Lemma~\ref{dualL2lifting}, we obtain the quasi-orthogonality for the limiting-dual problem.
\begin{lemma}[Quasi-orthogonality for Limiting Dual Problem]\label{dual_quasi}
Let the problem data satisfy Assumption~\ref{data_assumptions}, and $\cT_1, \cT_2 $ be two conforming triangulations with $\cT_2 \ge \cT_1$.  Let  $\hat z \in H_0^1(\Omega)$  the solution to ~\eqref{limiting_dual_problem} and $\hat z_i \in \V_i$ the solution to~\eqref{Ldiscrete_dual_problem}, $i = 1,2$. 
There exists a constant $\hat C_\ast > 0$ depending on the problem data $\bD$ and initial mesh $\cT_0$, and a number $0 < s \le 1$ related to the regularity of \eqref{linearized_adjoint}, such that for sufficiently small $h_0$ we have
\begin{equation}\label{Dquasi_ortho0}
\nrse{\hat z - \hat z_{2}}^2 \le \hat \Lambda \nrse{\hat z - \bar v}^2 - \nrse{\hat z_{2} - \bar v}^2, \quad \forall \bar v \in \V_{2},
\end{equation}
and in particular for $\bar v = \hat z_1$
\begin{equation}\label{Dquasi_ortho1}
\nrse{\hat z - \hat z_{2}}^2 \le \hat \Lambda_G \nrse{\hat z - \hat z_1}^2 - \nrse{\hat z_{2} - \hat z_1}^2
\end{equation}
where
\[
\hat \Lambda \coloneqq (1 -B m_\cE^{-1}\hat C_\ast  h_0^{s})^{-1} ~\an~ \hat \Lambda_G \coloneqq (1 - B \hat C_\ast^2 h_0^{2s})^{-1} 
\]
and $\hat C_\ast$ is the constant from Lemma~\ref{dualL2lifting}.
\end{lemma}
\begin{proof}
The proof follows same arguments as in Lemma~\ref{primal_quasi}, except that in place of the inequality in~\eqref{QP2} we have for the limiting dual problem
\begin{align}\label{DQP2}
a(u - u_2, v) + \langle b'(u)(\hat z  - \hat z_2), v \rangle = 0 \tforall v \in \V_2,
\end{align}
yielding 
\begin{align}\label{DQP3}
2a (\hat z - \hat z_2, \bar v -\hat z_2) & \le 2B \nr{\hat z - \hat z_2}_{L_2}  \nr{\bar v - \hat z_2}_{L_2},  
\end{align}
as in~\eqref{QP3}. The rest of the proof is similar to Lemma~\ref{primal_quasi}, and we omit it here.
\end{proof}
%
\subsection{Estimator Perturbations for Dual Sequence}\label{subsec:Dest}
As we have seen in Theorem~\ref{lemma:est_reduc}, the local Lipschitz property (cf.~Lemma~\ref{primalLL}) plays a key role in deriving the estimator reduction property used to convert between estimators on different refinement levels in both the primal and limiting dual problems. The following lemma gives similar local Lipschitz properties for the approximate and limiting dual problems on a given refinement level.  
\begin{lemma}[Local Lipschitz Property for Dual Estimators]\label{dualLL}
Let the problem data satisfy Assumption~\ref{data_assumptions} and Assumption~\ref{a:apriori}. Let $\cT$ be a conforming refinement of $\cT_{0}$. Then for all $T \in \cT$ and for any  $v, w \in \V_\cT$, it holds that
\begin{align}\label{Dperturb_est0}
|\zeta_{\cT,j}(v,T)  - \zeta_{\cT,j}(w,T) | & \le \bar \Lambda_1 \eta_\cT(\bD,T)\nr{v - w}_{H^1(\omega_T)}.
\end{align}
In particular, for the error indicator of the limiting dual problem we have
\begin{align}\label{Dperturb_est0A}
|\zeta_{\cT}(v,T) - \zeta_{\cT}(w,T)| & \le \bar \Lambda_1 \eta_\cT(\bD,T)\nr{v - w}_{H^1(\omega_T)}.
\end{align}
 The constant $\bar \Lambda_1> 0$ depends on the dimension $d$ and the regularity of the initial mesh $\cT_0$.
\end{lemma}
\begin{proof}
The proof is similar to Lemma~\ref{primalLL}. We sketch the proof below. To prove~\eqref{Dperturb_est0}, by~\eqref{zeta_cmpct} we have
\begin{equation}\label{DLLresult}
\zeta_{\cT,j}^2(v,T) \coloneqq h_T^2 \nr{\hat R_j^\ast(v)}_{L_2(T)}^2 +  h_T \nr{ J_T(v)    }_{L_2(\pa T)}^2,   \quad v \in \V_\cT.
\end{equation}
Setting $e = v - w$ and applying linearity  to the definition of the dual residual as given by~\eqref{dualresi}, we obtain
\begin{align*}
\hat R_j^\ast(v)  = g + \hat \cL_j^\ast( w + e)  = \hat R_j^\ast(w) + \hat \cL_j^\ast(e).
\end{align*}
By the same reasoning as~\eqref{L0}, we get
\begin{align}\label{DL0}
\zeta_{\cT,j}(v,T) & \le \zeta_{\cT,j}(w,T) + h_T\nr{\hat \cL_j^\ast (e)}_{L_2(T)} + h_T^\half \nr{J(e)}_{L_2(\pa T)}. 
\end{align}
The term $\hat \cL_j^\ast$ (respectively $\hat \cL^\ast$ for the limiting dual) in~\eqref{DL0} satisfies the same bound as the analogous term $\cD$ in ~\eqref{L0} of Lemma~\ref{primalLL}. Hence the bounds~\eqref{Dperturb_est0} and~\eqref{Dperturb_est0A} hold with the same constants as in~\eqref{perturb_est}.
\end{proof}

With the help of Lemma~\ref{dualLL}, we are able to derive the following corollary, which addresses the error induced by switching between error indicators corresponding to the  approximate and limiting dual problems on a given element.  
\begin{corollary}\label{consecIndicators}
Let the problem data satisfy Assumption~\ref{data_assumptions} and Assumption~\ref{a:apriori}. Let $\cT $ be a conforming refinement of $\cT_{0}$, and $u, u_{j}$ are the solutions to~\eqref{primal_problem} and \eqref{discrete_primal} problems, respectively.   Let $\Theta \an K_Z$ the constants given in Proposition~\ref{assumptionsonb}.  For all $T \in \cT$ and for $v,w \in \V_\cT \cap [z_-,z_+]$  the dual indicator on $\cT$ satisfies
\begin{equation}\label{Dperturb_est}
|\zeta_{\cT,j}(v,T) - \zeta_{\cT,k}(w,T)|\le \bar \Lambda_1 \eta_\cT(\bD,T)\nr{v - w}_{H^1(\omega_T)}
+ \Theta K_Z  h_T\nr{u_j - u_k}_{L_2(T)}.
\end{equation}
In particular, for $\cT = \cT_1$, we have for the limiting estimator
\begin{align}
|\zeta_{1,1}(v,T)  - \zeta_{1}(w,T)| &\le \bar \Lambda_1 \eta_1(\bD,T)\nr{v - w}_{H^1(\omega_T)}
+ \Theta K_Z  h_T\nr{u - u_1}_{L_2(T)}, \label{CP0p}.
\end{align}
\end{corollary}
\begin{proof}
By the definition of the residuals for the approximate dual problems,  for any $w\in \V_{\cT}$ we have
\begin{align}\label{CP1}
\hat R^\ast_j(w) &= g + \nabla \cdot (A \nabla w) + b'(u_k) w + \left(b'(u_j) - b'(u_k)  \right)w
\nonumber \\
& = \hat R^\ast_k(w) +  \left(b'(u_j) - b'(u_k)  \right)w.
\end{align}
Using~\eqref{CP1} in the definition of the dual indicator~\eqref{zeta_cmpct} and applying a generalized triangle inequality
\begin{align}\label{CP2}
\zeta_{\cT,j}(w,T) & = \left( h_T^2 \nr{\hat R^\ast_k(w) + (b'(u_j) - b'(u_k)) w}_{L_2(T)}^2  + h_T\nr{J_T(w)} _{L_2(\pa T)}^2 \right)^{1/2} 
\nonumber \\
& \le \left( h_T^2 \nr{\hat R_k^\ast(w)}_{L_2(T)}^2 + h_T \nr{J_T(w)}_{L_2(\pa T)}^2  \right)^{1/2} + h_T \nr{\left(b'(u_j) - b'(u_k)\right) w}_{L_2(T)} 
\nonumber \\
& \le \zeta_{\cT,k}(w,T) +  \Theta K_Z h_T\nr{u_j - u_k}_{L_2(T)}.
\end{align}
Applying~\eqref{Dperturb_est0} in Lemma~\ref{dualLL} to the estimate~\eqref{CP2}, we obtain~\eqref{Dperturb_est}.
\end{proof}

As an immediate consequence of Corollary~\ref{consecIndicators}, we have the following results on the error induced by switching between dual estimators over a collection of elements on a given refinement level.  This estimate plays a key role in the contraction argument below, as we apply it to switching between the estimator for the limiting dual and the computed error estimators for the approximate dual problems in the GOAFEM algorithm.
\begin{corollary}\label{cor:perturb_sets} 
Let the hypotheses of Corollary~\ref{consecIndicators} hold.  Then for any subsets $\cM_1, \cM_2 \subseteq \cT_1$ and arbitrary $\delta_1, \delta_2, \delta_A, \delta_B > 0$
\begin{align}\label{DPS0}
\zeta_1^2(v,\cM_1) &\ge (1 + \delta_1)^{-1}  (1 + \delta_A)^{-1}\zeta_{1,1}^2(w,\cM_1) 
\nonumber \\
& \quad -  (1 + \delta_1)^{-1} \delta_A^{-1} \Theta^2 K_Z^2 h_0^2\nr{u - u_1}_{L_2}^2 
-(d+2)\delta_1^{-1} \bar \Lambda_1^2 \eta_0^2 \nr{v - w}_{H^1}^2 
\\ 
\zeta_{1,1}^2(w,\cM_2) &\ge (1 + \delta_2)^{-1}  (1 + \delta_B)^{-1}\zeta_{1}^2(v,\cM_2) 
\nonumber \\ \label{DPS1}
& \quad -  (1 + \delta_2)^{-1} \delta_B^{-1} \Theta^2 K_Z^2 h_0^2\nr{u - u_1}_{L_2}^2 
-(d+2)\delta_2^{-1} \bar \Lambda_1^2 \eta_0^2 \nr{v - w}_{H^1}^2. 
\end{align}
\end{corollary}
\begin{proof}
The conclusions follow by squaring inequality~\eqref{CP0p}, applying Young's inequality twice, and then summing over element $T \in \cM_{1}$ (respectively $T\in \cM_{2}$).  The $H^1$ norm is summed over all elements $T \in \cT_1$ counting each element $d+2$ times, the maximum number of elements in each patch $\omega_T$. 
\end{proof}


\subsection{Contraction of GOAFEM}\label{subsec:contraction}
The main contraction argument Theorem~\ref{quasi_combine_C} follows after two more lemmas. The first combines a sequence of estimates to convert the non-computable limiting estimator for the dual problem to a computable quantity, apply the D\"orfler property and then convert back. The second relates the difference between the Galerkin solutions of the limiting and approximate dual problems to the primal error.  
Motivated by estimator reduction for the limiting dual problem as in  equation~\eqref{Dest_reduc}
\begin{align}
\zeta_{2}^2(\hat z_2, \cT_2)  \le  &(1 + \delta) \left\{ \zeta_{1}^2(\hat z_1,\cT_1) - \lambda \zeta_{1}^2(\hat z_1,\cM) \right\}  + ( 1 + \delta^{-1}) \Lambda_1 \eta_0^2 \nrse{\hat z_2 - \hat z_1}^2,
\end{align}
the following lemma addresses the conversion between the limiting estimator $\zeta_{1}^2(\hat z_1,\cM)$ and and the computable estimator $\zeta_{1,1}^2(\hat z_1^1,\cM)$ necessary for marking the mesh for refinement.  
\begin{lemma}\label{estimatorSwitch}
Let the problem data satisfy Assumption~\ref{data_assumptions} and Assumption~\ref{a:apriori}. Let $\Theta \an K_Z$ as given by Proposition~\ref{assumptionsonb}, $C_\ast$ as given by Lemma~\ref{L2lifting} and $\Lambda_1$ as given in Lemma~\ref{lemma:est_reduc}. Let
\begin{align*}
u & \text{ the solution to}~\eqref{primal_problem}, 
&u_1 & \text{ the solution to}~\eqref{discrete_primal}, \\
\hat z & \text{ the solution to}~\eqref{limiting_dual_problem},
& \hat z_1 & \text{ the solution to}~\eqref{Ldiscrete_dual_problem},
&\hat z_1^1 & \text{ the solution to}~\eqref{approx_dual_problem}.
\end{align*}
Let $\zeta_{1,1}(\hat z_1^1, \cM)$ satisfy the D\"orfler property \eqref{markD} for $\cM \subset \cT_1$, namely
$
\zeta_{1,1}^2(\hat z_1^1, \cM) \ge \theta^2 \zeta^2_{1,1}(\hat z_1^1, \cT_1).
$ 
Then for arbitrary $\delta_1, \delta_2, \delta_A, \delta_B > 0$ there is a $\delta_4>0$   such that
\begin{align}\label{ES0}
-\zeta_1^2(\hat z_1,\cM) 
&\le - \f {\beta \theta^2}{ (1 + \delta_4)} \zeta_{1}^2(\hat z_1,\cT_1) 
- \f{(1-\beta) \theta^2}{ (1 + \delta_4)C_{1}^{2}} \nrse{\hat z - \hat z_1}^2
\nonumber \\
& \quad +  \left( \f{ \theta^2}{  (1 + \delta_A) (1 + \delta_2) \delta_B} + \f 1\delta_A   \right)  \f{  \Theta^2 K_Z^2 C_\ast^2 h_0^{2(1+s)}}{(1 + \delta_1)}  \nrse{u - u_1}^2 
\nonumber \\
& \quad + \left( \f{\theta^2}{ (1 + \delta_1)  (1 + \delta_A) \delta_2} + \f 1 \delta_1 \right)\Lambda_1 \eta_0^2(\bD, \cT_0) \nrse{\hat z_1 -\hat  z_1^1}^2.
\end{align}
\end{lemma}
\begin{proof}
From Corollary~\ref{cor:perturb_sets}, $L_2$-lifting Lemma~\ref{L2lifting} and coercivity~\eqref{coercive}
\begin{align}\label{ES3}
-\zeta_1^2(\hat z_1,\cM) &\le -(1 + \delta_1)^{-1}  (1 + \delta_A)^{-1}\zeta_{1,1}^2(\hat z_1^1,\cM) 
\nonumber \\
& \quad +  (1 + \delta_1)^{-1} \delta_A^{-1} \Theta^2 K_Z^2 h_0^2\nr{u - u_1}_{L_2}^2 
+\delta_1^{-1} \bar \Lambda_1^2 (d+2)\eta_0^2 \nr{\hat z_1 -\hat  z_1^1}_{H^1}^2 
\nonumber \\
&\le -(1 + \delta_1)^{-1}  (1 + \delta_A)^{-1}\zeta_{1,1}^2(\hat z_1^1,\cM) 
\nonumber \\
& \quad +  (1 + \delta_1)^{-1} \delta_A^{-1} \Theta^2 K_Z^2 C_\ast^2 h_0^{2(1+s)}\nrse{u - u_1}^2
+\delta_1^{-1} \Lambda_1 \eta_0^2 \nrse{\hat z_1 -\hat  z_1^1}^2,
\end{align}
with $\Lambda_1 \coloneqq \bar \Lambda_1^2(d+2)m_\cE^{-2}$.
The D\"orfler property may be applied to the first term on the RHS of~\eqref{ES3}
\begin{equation}\label{ES4}
 - \zeta_{1,1}^2(\hat z_1^1,\cM)  \le  - \theta^2 \zeta_{1,1}^2(\hat z_1^1).
\end{equation}
Converting back to he limiting estimator by~\eqref{DPS1} in Corollary~\ref{cor:perturb_sets}
\begin{align}\label{ES5}
-\zeta_{1,1}^2(\hat z_1^1)&\le -(1 + \delta_2)^{-1}  (1 + \delta_B)^{-1}\zeta_{1}^2(\hat z_1, \cM) 
\nonumber \\
& \quad +  (1 + \delta_2)^{-1} \delta_B^{-1} \Theta^2 K_Z^2 C_\ast^2 h_0^{2(1+s)}\nrse{u - u_1}^2
+\delta_2^{-1} \Lambda_1 \eta_0^2 \nrse{\hat z_1 -\hat  z_1^1}^2.
\end{align}
Define $\delta_4$ by
\begin{equation}\label{ES8}
(1 + \delta_4) \coloneqq (1 + \delta_1)(1 + \delta_2)(1 + \delta_A)(1 + \delta_B).
\end{equation}
Then by plugging \eqref{ES4} and \eqref{ES5} in the first term on the RHS of~\eqref{ES3}, we obtain
\begin{align}\label{ES6}
-\zeta_1^2(\hat z_1,\cM) 
&\le - \theta^2 (1 + \delta_4)^{-1} \zeta_{1}^2(\hat z_1) 
\nonumber \\
& \quad +  \left( \theta^2  (1 + \delta_A)^{-1} (1 + \delta_2)^{-1} \delta_B^{-1} + \delta_A^{-1}   \right) (1 + \delta_1)^{-1}  \Theta^2 K_Z^2 C_\ast^2 h_0^{2(1+s)}\nrse{u - u_1}^2 
\nonumber \\
& \quad + \left( \theta^2 (1 + \delta_1)^{-1}  (1 + \delta_A)^{-1} \delta_2^{-1} + \delta_1^{-1} \right)\Lambda_1 \eta_0^2 \nrse{\hat z_1 -\hat  z_1^1}^2.
\end{align}
Finally, we split the first term on the RHS of~\eqref{ES6} into two pieces for some $\beta \in (0,1)$,  and apply the upper-bound estimate \eqref{dual_upper} in Lemma~\ref{primal_upper} to the second piece yielding
\begin{align}\label{ES7}
-\zeta_1^2(\hat z_1,\cM) 
&\le - \beta \theta^2 (1 + \delta_4)^{-1} \zeta_{1}^2(\hat z_1) 
 - (1-\beta) \theta^2 (1 + \delta_4)^{-1} C_1^{-2}\nrse{\hat z - \hat z_1}^2
\nonumber \\
+ & \left( \theta^2  (1 + \delta_A)^{-1} (1 + \delta_2)^{-1} \delta_B^{-1} + \delta_A^{-1}   \right) (1 + \delta_1)^{-1}  \Theta^2 K_Z^2 C_\ast^2 h_0^{2(1+s)}\nrse{u - u_1}^2 
\nonumber \\
+ & \left( \theta^2 (1 + \delta_1)^{-1}  (1 + \delta_A)^{-1} \delta_2^{-1} + \delta_1^{-1} \right)\Lambda_1 \eta_0^2 \nrse{\hat z_1 -\hat  z_1^1}^2.\nonumber
\end{align}
This completes the proof.
\end{proof}

We may convert $\nrse{\hat z_1 -\hat  z_1^1}$ in the last term on the RHS of~\eqref{ES0}  to the error $\nrse{u -u_1}$ as stated in the following lemma.
\begin{lemma}\label{DPErr}
Let the problem data satisfy Assumption~\ref{data_assumptions} and Assumption~\ref{a:apriori}. Let $\Theta \an K_Z$ the constants given in Proposition~\ref{assumptionsonb} and $C_\ast \an \hat C_\ast$ the constants given by Lemmas~\ref{L2lifting} and~\ref{dualL2lifting}, respectively. Let
\begin{align*}
u & \text{ the solution to}~\eqref{primal_problem}, 
&u_1 & \text{ the solution to}~\eqref{discrete_primal}, \\
 \hat z_1 & \text{ the solution to}~\eqref{Ldiscrete_dual_problem},
&\hat z_1^1 & \text{ the solution to}~\eqref{approx_dual_problem}.
\end{align*}
Then
\begin{equation}\label{DP0}
\nrse{\hat z_1 - \hat z_1^1} \le \Theta  K_Z C_\ast \hat C_\ast h_0^{2s} \nrse{u - u_1}.
\end{equation}
\end{lemma}
\begin{proof}
Recall that
\begin{align}\label{DP1}
\hat z_1 & \text{ solves } a(\hat z_1, v) + \langle b'(u) \hat z_1, v\rangle = g(v), ~ \tforall v \in \V_1, \\
\label{DP2}
\hat z_1^1 & \text{ solves } a(\hat z_1^1, v) + \langle b'(u_1) \hat z_1^1, v\rangle = g(v), ~ \tforall v \in \V_1. 
\end{align}
Subtracting~\eqref{DP2} from~\eqref{DP1} and rearranging terms, we get
\begin{equation}\label{DP3}
a(\hat z_1 - \hat z_1^1, v) + \langle (b'(u) - b'(u_1)) \hat z_1, v\rangle  =  \langle b'(u_1)(\hat z_1^1 - \hat z_1), v\rangle, ~v \in \V_1.
\end{equation}
In particular, for $v = \hat z_1 - \hat z_1^1 \in \V_1$ equation~\eqref{DP3} yields
\begin{align}\label{DP4}
\nrse{\hat z_1 - \hat z_1^1}^2 &=- \langle (b'(u) - b'(u_1)) \hat z_1, \hat z_1 - \hat z_1^1\rangle  -  \langle b'(u_1)(\hat z_1 - \hat z_1^1), \hat z_1 - \hat z_1^1\rangle
\nonumber \\
& \le - \langle (b'(u) - b'(u_1)) \hat z_1, \hat z_1 - \hat z_1^1\rangle,
\end{align}
where in the last inequality, we used the monotonicity assumption of $b$ in Assumption~\eqref{data_assumptions}.  Now applying the Lipschitz property of $b'$, the \iic{a priori} $L_{\infty}$ bounds on the dual solution $\hat z_1$ (cf. Proposition~\ref{assumptionsonb}),  and both primal and dual $L_2$ lifting in \eqref{DP4}, we obtain
\begin{align}\label{DP5}
\nrse{\hat z_1 - \hat z_1^1}^2 & \le \Theta K_Z \nr{u - u_1}_{L_2} \nr{\hat z_1 - \hat z_1^1}_{L_2}
\nonumber \\
& \le \Theta K_Z C_\ast \hat C_\ast  h_0^{2s}\nrse{u - u_1} \nrse{\hat z_1 - \hat z_1^1},
\end{align}
from which the result follows.
\end{proof}

Now we are in position to show the contraction of GOAFEM in terms of the combined quasi-error which is a linear combination of the energy errors and error estimators in primal and limiting dual problems.  

\begin{theorem}[Contraction of GOAFEM]
\label{quasi_combine_C}
Let the problem data satisfy Assumption~\ref{data_assumptions} and Assumption~\ref{a:apriori}. Let 
\begin{align*}
u & \text{ the solution to}~\eqref{primal_problem}, 
&u_j & \text{ the solution to}~\eqref{discrete_primal}, \\
\hat z & \text{ the solution to}~\eqref{limiting_dual_problem},
& \hat z_j & \text{ the solution to}~\eqref{Ldiscrete_dual_problem}.
\end{align*}
Let $\theta \in (0,1]$, and let $\{\cT_j, \V_j\}_{j \ge 0}$ be the sequence of meshes and finite element spaces produced by GOAFEM.  Let $\gamma_p > 0$ as given by Theorem~\ref{primal_contraction}. Then for sufficient small mesh size $h_{0}$, there exist constants $\gamma > 0, \pi > 0 \an \alpha_D \in (0, 1)$ such that
\begin{align}\label{DC0}
& \nrse{\hat z - \hat z_2}^2  + \gamma \zeta_2^2(\hat z_2) + \pi\nrse{u - u_2}^2 + \pi \gamma_p \eta_2^2(u_2) 
\nonumber \\ 
& \qquad \le  \alpha_D^2\left(  \nrse{\hat z - \hat z_1}^2 + \gamma  \zeta_1^2(\hat z_1) + \pi \nrse{u - u_1}^2 + \pi \gamma_p \eta_1^2(u_1) \right).
\end{align}
\end{theorem}

\begin{proof} 
For simplicity, we denote 
$
\eta_0 = \eta_0(\bD, \cT_0) \an \zeta_{k}(\hat z_k) = \zeta_{k}(\hat z_k, \cT_k), ~ k = 1, 2.
$ 
By the estimator reduction for the limiting dual problem~\eqref{Dest_reduc}, for arbitrary $\delta > 0$ we have 
\begin{align}\label{AC1}
\zeta_{2}^2(\hat z_2)  \le  &(1 + \delta) \left\{ \zeta_{1}^2(\hat z_1) - \lambda \zeta_{1}^2(\hat z_1,\cM) \right\}  + ( 1 + \delta^{-1}) \Lambda_1 \eta_0^2 \nrse{\hat z_2 - \hat z_1}^2,
\end{align}
where $\lambda= 1- 2^{-1/d}.$ 
Recall the quasi-orthogonality estimate in the limiting dual problem from Lemma~\ref{dual_quasi}
\begin{align}\label{AC2}
\nrse{\hat z - \hat z_2}^2 \le \hat \Lambda_G \nrse{\hat z - \hat z_1}^2 - \nrse{\hat z_2 - \hat z_1}^2.
\end{align}
Adding~\eqref{AC2} to a positive multiple $\gamma$ (to be determined) of~\eqref{AC1} and applying the results of Lemmas~\ref{estimatorSwitch} and~\ref{DPErr} obtain
\begin{align}\label{AC3}
\nrse{\hat z - \hat z_2}^2 + \gamma \zeta_{2}^2(\hat z_2) & \le A\nrse{\hat z - \hat z_1}^2
+\gamma M \zeta_1^2(\hat z_1) + D \nrse{u - u_1}^2 
\nonumber \\
& \quad + \left( \gamma(1 +\delta^{-1}) \Lambda_1 \eta_0^2 -1 \right)\nrse{\hat z_2 - \hat z_1}^2.
\end{align}
We first set $\gamma \coloneqq (1 + \delta^{-1})^{-1} \Lambda_1^{-1}\eta_0^{-2}$ to eliminate the last term in ~\eqref{AC3}. This yields 
\begin{align}\label{AC4}
\nrse{\hat z - \hat z_2}^2 + \gamma \zeta_{2}^2(\hat z_2) & \le A\nrse{\hat z - \hat z_1}^2
+\gamma M \zeta_1^2(\hat z_1) + D \nrse{u - u_1}^2,
\end{align}
where the coefficients $A \an M$ of~\eqref{AC4} are given by
\begin{align}\label{AC5}
A & = \hat \Lambda_G - (1-\beta) \lambda  \theta^2 \delta  (1 + \delta_4)^{-1} C_1^{-2} \Lambda_1^{-1} \eta_0^{-2},
\\ \label{AC6}
M & = (1 + \delta)(1 -\beta  \lambda  \theta^2 (1 + \delta_4)^{-1} ),
\end{align}
where $\delta_{4}$ satisfies 
$
(1 + \delta_4 )\coloneqq (1 + \delta_1 )(1 + \delta_2 )(1 + \delta_A )(1 + \delta_B )
$ 
as was given in~\eqref{ES8}.

For contraction, we require $A<1$ and $M<1$ for the coefficients defined by~\eqref{AC5} and~\eqref{AC6}, that is, we need to choose a $\beta\in (0, 1)$ such that
\begin{equation}\label{AC8}
\f{\delta}{1 + \delta} \f{1 + \delta_4}{\lambda \theta^2} < \beta < 1 - \f{(\hat \Lambda_G - 1) \Lambda_C}{\delta} \f{1 + \delta_4}{\lambda \theta^2},
\end{equation}
with $\Lambda_C \coloneqq C_1^2 \Lambda_1 \eta_0^2.$
To demonstrate the existence of such a $\beta$, set
\begin{equation}\label{AC9}
\delta_4 = \delta = \frac{1}{2}\lambda \theta^2. 
\end{equation}
Then we require the mesh size $h_{0}$ sufficiently small, such that 
\begin{align}\label{meshCondition2}
\hat \Lambda_G < 1+ \frac{\lambda^{2}\theta^{4}}{2(2+\lambda \theta^{2})\Lambda_{C}},
\end{align}
for the given $\theta \in(0,1)$. Note the conditions \eqref{AC9} and \eqref{meshCondition2} guarantee that the interval in \eqref{AC8} is nonempty, so there exists a $\beta$ such that 
$$\frac{1}{2} < \beta  <  1 - \f{(\hat \Lambda_G - 1) \Lambda_C}{\lambda \theta^2} \left(1+ \f{2}{\lambda \theta^2}\right).$$

It remains to control the last term in \eqref{AC4}. For simplicity, we assume $\delta_1 = \delta_2 = \delta_A = \delta_B \eqqcolon \delta_C$. Then the coefficient $D$ in \eqref{AC4} is given by
\begin{align}\label{AC10}
D & = \delta \lambda  \Theta^2 K_Z^2 C_\ast^2 h_0^{2s} \left( \f{ \theta^2  +(1+ \delta_C)^2} {(1 + \delta_C)^2 \delta_C} \right) \left( \f{h_0^2}{\Lambda_1 \eta_0^2 (1+\delta_{C})} + \hat C_\ast^2 h_0^{2s} \right).
\end{align}
To control the primal error term with the coefficient $D$ as given by~\eqref{AC10}, we add a positive multiple $\pi$ (to be determined) of the primal contraction result~\eqref{contractionres} of Theorem~\ref{primal_contraction} to \eqref{AC5} yielding
\begin{align}\label{AC11}
& \nrse{\hat z - \hat z_2}^2  + \gamma \zeta_2^2(\hat z_2) + \pi\nrse{u - u_2}^2 + \pi \gamma_p \eta_2^2(u_2) 
\nonumber \\ 
&\qquad \le  A \nrse{\hat z - \hat z_1}^2 + \gamma M \zeta_1^2(\hat z_1) + (D + \alpha^2 \pi) \nrse{u - u_1}^2 + \alpha^2 \pi \gamma_P \eta_1^2(u_1).
\end{align}
We choose $\pi$ to ensure $D + \alpha^2 \pi < \pi$, namely,
$\pi > \f{D}{1 - \alpha^2}$,
and set
\begin{equation}\label{AC13}
\alpha_D^2 \coloneqq \max\left\{ A, M, \f{D + \alpha^2 \pi}{\pi}, \alpha^2 \right\} < 1.
\end{equation}
Then the combined quasi-error satisfies the contraction property \eqref{DC0}.
\end{proof}

For simplicity, we denote by 
\[
\bar Q^2(u_j, \hat z_j)  = \nrse{\hat z - \hat z_j}^2  + \gamma \zeta_j^2(\hat z_j) + \pi\nrse{u - u_j}^2 + \pi \gamma_p \eta_j^2(u_j) 
\]
the combined quasi-error in \eqref{DC0}. The following corollary gives the contraction of the error in the goal function, which is determined by the contraction of the combined quasi-error.
\begin{corollary}\label{goalErrorBound}
Let the assumptions in Theorem~\ref{quasi_combine_C} hold. 
Then the error in the goal function is controlled by a constant multiple of the square of the combined quasi-error, i.e., 
\begin{equation}\label{CGo0}
|g(u) - g(u_j)| \le C \bar Q^2_j(u_j, \hat z_j) \le \alpha_D^{2j} C \bar Q_0^2(u_0, \hat z_0).
\end{equation}
\end{corollary}
\begin{proof}
Choosing the test function $v = u - u_{j}$ in \eqref{limiting_dual_problem}, and by linearity and Galerkin orthogonality for the primal problem, we obtain
\begin{align}\label{CGo1}
g(u) - g(u_j) & = a(\hat z,u) + \langle b'(u) \hat z, u \rangle - a(\hat z, u_j) - \langle b'(u) \hat z, u_j\rangle 
\nonumber \\
& = a(u - u_j, \hat z) + \langle b'(u)(u - u_j), \hat z \rangle 
\nonumber \\
& = a(u - u_j, \hat z) + \langle \cB_j(u - u_j), \hat z \rangle + \langle  (b'(u) - \cB_j)(u - u_j), \hat z
\rangle
\nonumber \\
&  = a(u - u_j, \hat z - \hat z_j) + \langle b(u) - b(u_j), \hat z - \hat z_j \rangle + \langle  (b'(u) - \cB_j)(u - u_j), \hat z \rangle.
\end{align}
The third term in the last line of~\eqref{CGo1} represents the error induced by switching from \eqref{limiting_dual_problem} to \eqref{linearized_dual}.  This term may be bounded in terms of the constants and $L_\infty$ estimates in Proposition~\ref{assumptionsonb} and
\[
\nr{b'(u) - \cB_j}_{L_2} = \left\| { \int_0^1 b'(u) - b' \left(u_j + \xi(u - u_j) \right)d\xi  }\right\|_{L_2}  \le \f \Theta 2\nr{u - u_j}_{L_2},
\]
yielding
\begin{align}\label{CGo2}
\langle  (b'(u) - \cB_j)(u - u_j), \hat z \rangle
& \le K_Z \nr {b'(u) - \cB_j}_{L_2}\nr{u - u_j}_{L_2}
\nonumber \\
& \le \f 1 2 \Theta K_Z  \nr{u - u_j}_{L_2}^2.
\end{align} 
Then by~\eqref{CGo1},~\eqref{CGo2}, the Cauchy-Schwarz inequality and $L_2$-lifting as in Lemmas~\ref{L2lifting} and~\ref{dualL2lifting} 
\begin{align}\label{CGo3}
|g(u) - g(u_j)| & \le \nrse{u - u_j}\nrse{\hat z - \hat z_j} + B\nr{u - u_j}_{L_2}\nr{\hat z - \hat z_j}_{L_2} 
 +\f 1 2 \Theta K_Z  \nr{u - u_j}_{L_2}^2 
\nonumber \\
& \le (1 + BC_\ast \hat C_\ast h_0^{2s}) \nrse{u - u_j}\nrse{\hat z - \hat z_j}  + \f 1 2 \Theta K_Z  C_\ast^2   h_0^{2s} \nrse{u - u_j}^2 
\nonumber \\
& \le \f 1 2 \left( 1+(\Theta K_Z  C_\ast   + B \hat C_\ast )C_\ast h_0^{2s} \right) \nrse{u - u_j}^2 
+  \f 1 2(1 + BC_\ast \hat C_\ast h_0^{2s}) \nrse{\hat z - \hat z_j}^2.  
\end{align}
Therefore the error in the goal function is bounded above by a constant multiple of the square of the combined quasi-error $\bar Q^2(u_j, \hat z_j)$. Thus~\eqref{CGo0} follows by the contraction result in Theorem~\ref{quasi_combine_C}.
\end{proof}

\section{Numerical Experiments}
\label{sec:num}
In this section, we present some numerical experiments implemented 
using FETK~\cite{Hols2001a},
which is a fairly standard set of finite element modeling libraries for 
approximating the solutions to systems of nonlinear elliptic and parabolic 
equations.  We compare three methods: HPZ, the algorithm presented in this paper; MS, the algorithm presented in~\cite{MoSt09}; and the DWR, the dual weighted residual method as described in, for example~\cite{BaRa03,Becker.R;Rannacher.R1996a,EHL02,Giles.M;Suli.E2003,Gratsch.T;Bathe.K2005,EHM01}.  We see HPZ performs with comparable efficiency to MS, with the added benefit of fewer iterations of the adaptive algorithm~\eqref{goafem00} resulting in a shorter overall runtime.  The efficiency of the residual based algorithms HPZ and MS in comparison to DWR varies with the problem structure.  The examples below show cases where each algorithm may outperform the others, but where the performance of all three is comparable with a small change in the problem parameters.

 In our DWR implementation, the finite element space $\V_{k}$ for the primal problem 
 employs linear Lagrange elements as do HPZ and MS for both the primal and dual spaces.  For DWR, the dual finite element space $\V_{k}^2$ uses quadratic Lagrange elements. The elementwise DWR indicator defined as: 
\begin{equation*}
\eta_{k}^D (v,T) \coloneqq \langle R(v), z^2 - I_k z^2\rangle_T  + \f 1 2  \langle J_T(v), z^2 - I_k z^2 \rangle_{\pa T},   \quad v \in \V_{k}
\end{equation*}
estimates the influence of the dual solution on the primal residual.
Here $z^2 \in \V_{k}^2$ is the solution of the approximate dual problem \eqref{approx_dual_problem} and $I_k$ is the interpolator onto $\V_{k}$. Then the DWR error estimator is  the absolute value of the sum of indicators  
\begin{align*}
\eta_k^D =  \left| \sum_{T \in \cT_k} \eta_T^D(u_k,T) \right| \le\sum_{T \in \cT_k}  \left| \eta_T^D(u_k,T)\right|.  
\end{align*}
Both HPZ and MS use the residual based indicators \eqref{eta_cmpct} and \eqref{zeta_cmpct} for primal and approximate dual problems, respectively.

In the adaptive algorithms, we use the D\"orfler marking strategy \eqref{markP}-\eqref{markD} with  parameter $\theta = 0.6.$ For the nonlinear primal problem, at each refinement we use 
a Newton-type iteration to solve the resulting nonlinear system of algebraic
equations, which reduces the nonlinear residual to the tolerance 
$ \|F(u)\|_{L_{2}} \le 10^{-7}$. 
On the initial triangulation, we use a zero initial guess for the Newton 
iteration; then for each subsequent refinement, we interpolate the numerical 
solution from the previous step to the current triangulation and then
use it as the initial guess for the Newton iteration. By doing this, we have a good initial guess for the Newton iteration indicating a quadratic convergence rate of the nonlinear iterations. 

In the following examples, we use the same primal problem given in weak form by
\[
\f 1{1000}\langle  \nabla u, \nabla v \rangle+ \langle 3 u^3, v \rangle = f(v).
\]
The problem data for each problem are defined by $g(u) = \int_\Omega g u$ and $f(v) = \int_\Omega f v$ where $g= g(x,y)$, $f=f(x,y)$  are defined in each example over the domain $\Omega = (0,1)^2$. The initial triangulation is a uniform mesh consisting of 144 elements.  Here we consider problems where the primal and dual data and likewise the primal and dual solutions contain either sharp spikes or shallower bumps where these functions feature rapidly changing gradients.

\begin{example}[Separated primal data] 
\end{example}
This problem features a single Gaussian spike as the goal function $g(x,y)$ and primal data focused on two bumps, one of which overlaps with the spike in $g(x,y)$.  We look at four sets of parameters manipulating both the intensity of the Gaussian and the placement of the second primal bump.  This problem demonstrates the difference between the algorithms when some or all of the primal data has a strong influence on the quantity of interest $g(u)$ and is remote from the spike in the dual solution. 

The goal function is given by $g(x,y) = a \exp(-a [(x - x_d)^2 + (y - y_d)^2])$.  The primal data $f(x,y)$ is chosen so the exact solution $u$ is 
\[
u(x,y) = \sin(2\pi x)\sin(2\pi y) \left\{  \f 1  {2[(x-x_0)^2+(y-y_0)^2] + 10^{-2}) } + \f 1 {2[(x-x_1)^2+(y-y_1)^2]  + 10^{-2} }  \right\}.
\]
The fixed parameters $(x_0,y_0) = (0.7, 0.7)$ and $(x_d, y_d) = (0.7, 0.7)$ fix an interaction between the primal solution and the dual data, and the second primal spike $(x_1,y_1)$ is tested at a near and far location.  The parameter $a$ scales both the maximum intensity of the spike in the goal function as well as the spread of the influence of the dual solution.
\begin{align}
Figure~\ref{fig:HSD_266}  && (x_1,y_1) = (0.6,0.6)  &&  a = 200 \label{HSD_266}, \\
				        && (x_1,y_1) = (0.6,0.6)  &&  a = 400 \label{HSD_466}. \\
Figure~\ref{fig:HSD_233}  && (x_1,y_1) = (0.3,0.3)  &&  a = 200 \label{HSD_233}, \\
				         && (x_1,y_1) = (0.3,0.3)  &&  a = 400 \label{HSD_433}.
\end{align}

\begin{figure}
\includegraphics[width=0.45\textwidth]{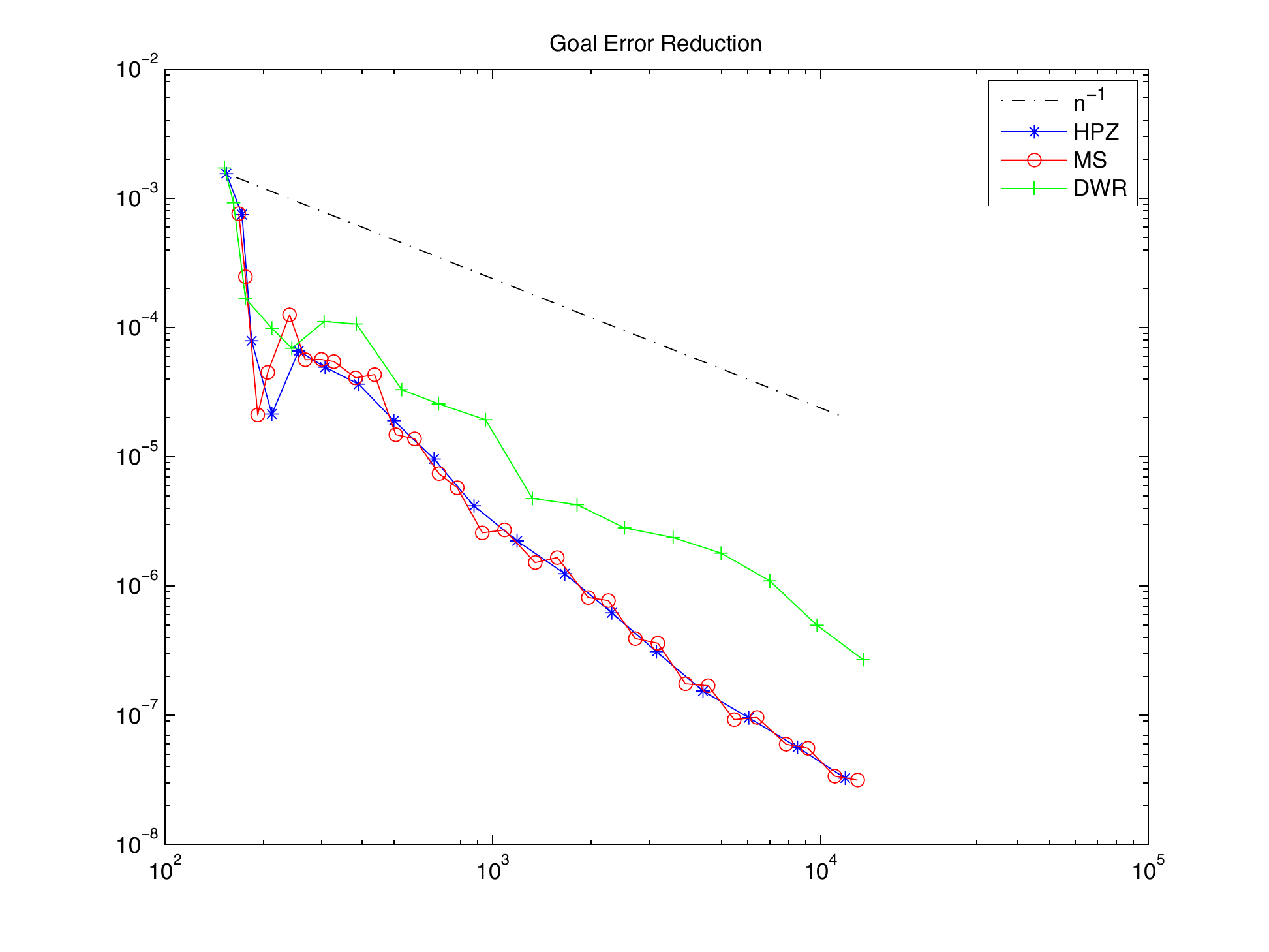}~
\includegraphics[width=0.45\textwidth]{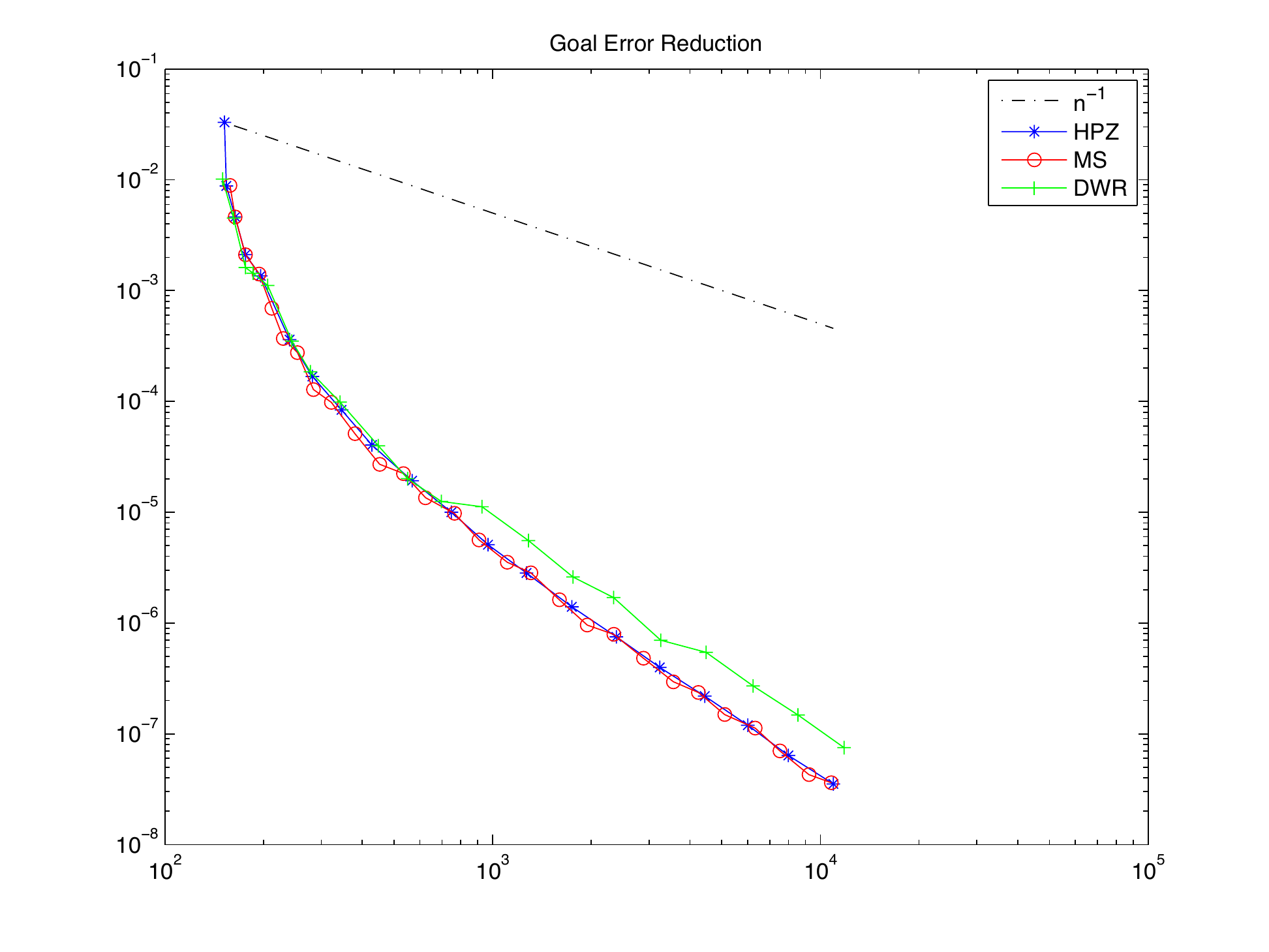}
\caption{Left: goal error after 19 HPZ, 33 MS and 19 DWR iterations for ~\eqref{HSD_266}, 
compared with $n^{-1}$.  Right: goal error after 21 HP, 31 MS and 21 DWR iterations for ~\eqref{HSD_466}, compared with $n^{-1}$. }
\label{fig:HSD_266}
\end{figure}

In Figure~\ref{fig:HSD_266} one spike in the primal data is focused near $(0.7,0.7)$ overlapping with the spike in $g(x,y)$ and the second is near $(0.6,0,6)$, close enough to influence $u$ in the vicinity of $g$, but not entirely overlapping with the spike in the dual solution.  In these cases, the residual based methods outperform DWR when $a= 200$ in parameter set~\eqref{HSD_266} but only slightly when $a=400$ in~\eqref{HSD_466} where the spike in $g(x,y)$ is narrowed and the second primal solution spike has less influence on the quantity of interest $g(u)$.

\begin{figure}[htp]
\includegraphics[width=0.45\textwidth]{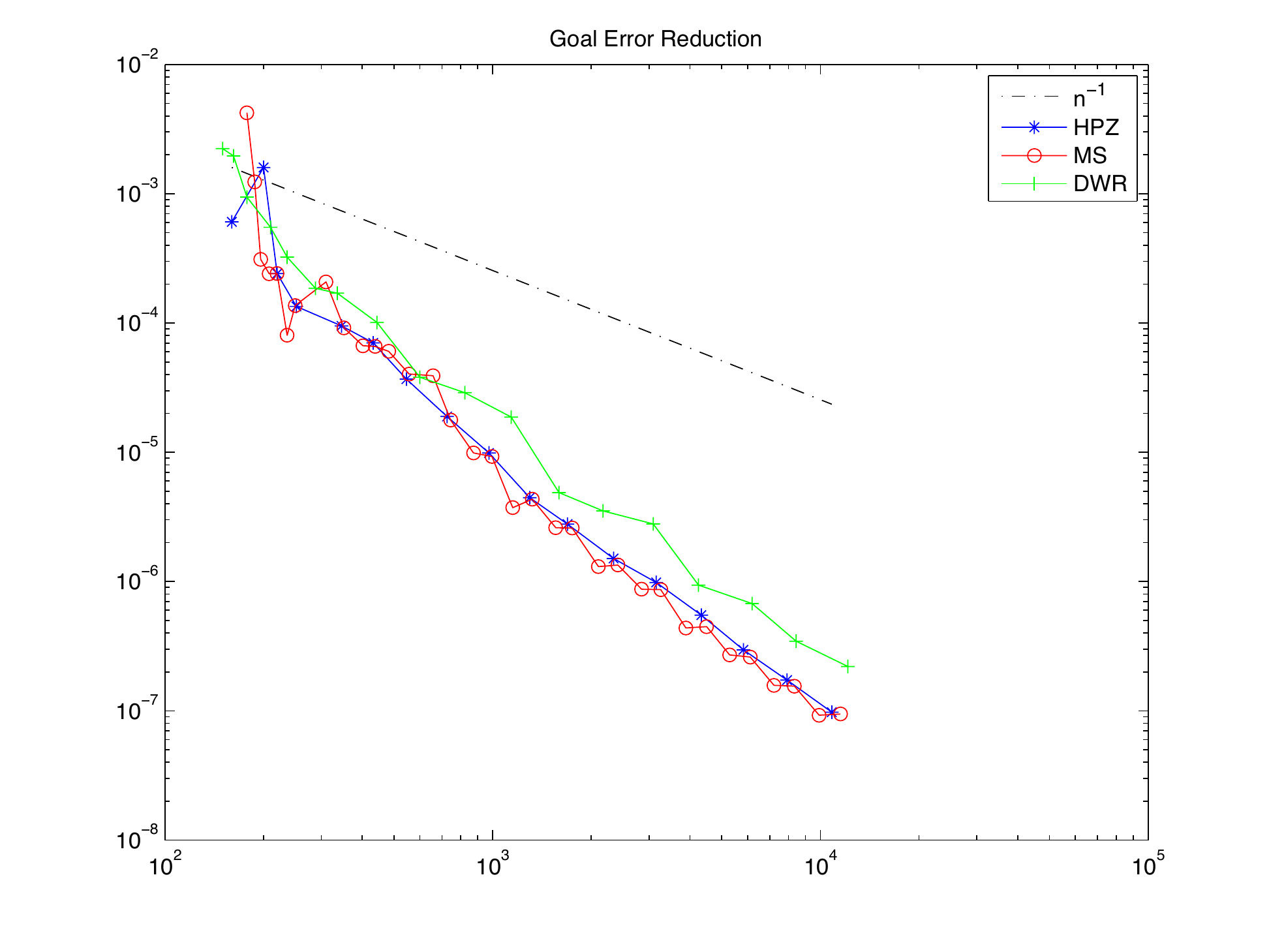}~
\includegraphics[width=0.45\textwidth]{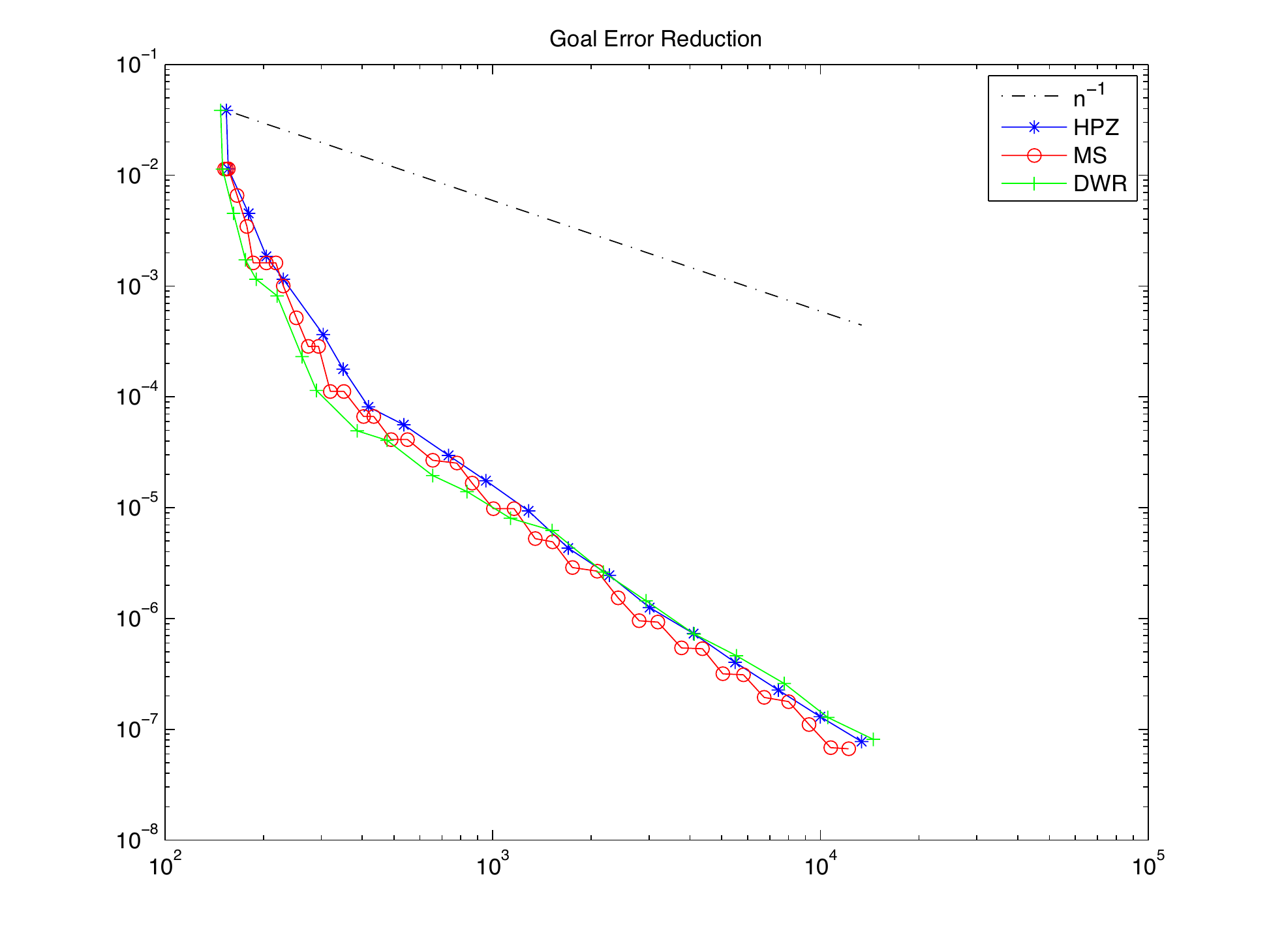}
\caption{Left: goal error after 18 HPZ, 36 MS and 19 DWR iterations for parameter set~\eqref{HSD_233}, compared with $n^{-1}$.  Right: goal error after 21 HPZ, 42 MS and 22 DWR iterations for parameter set~\eqref{HSD_433}, compared with $n^{-1}$. }
\label{fig:HSD_233}
\end{figure}

\begin{figure}[htp]
\includegraphics[width=0.3\textwidth]{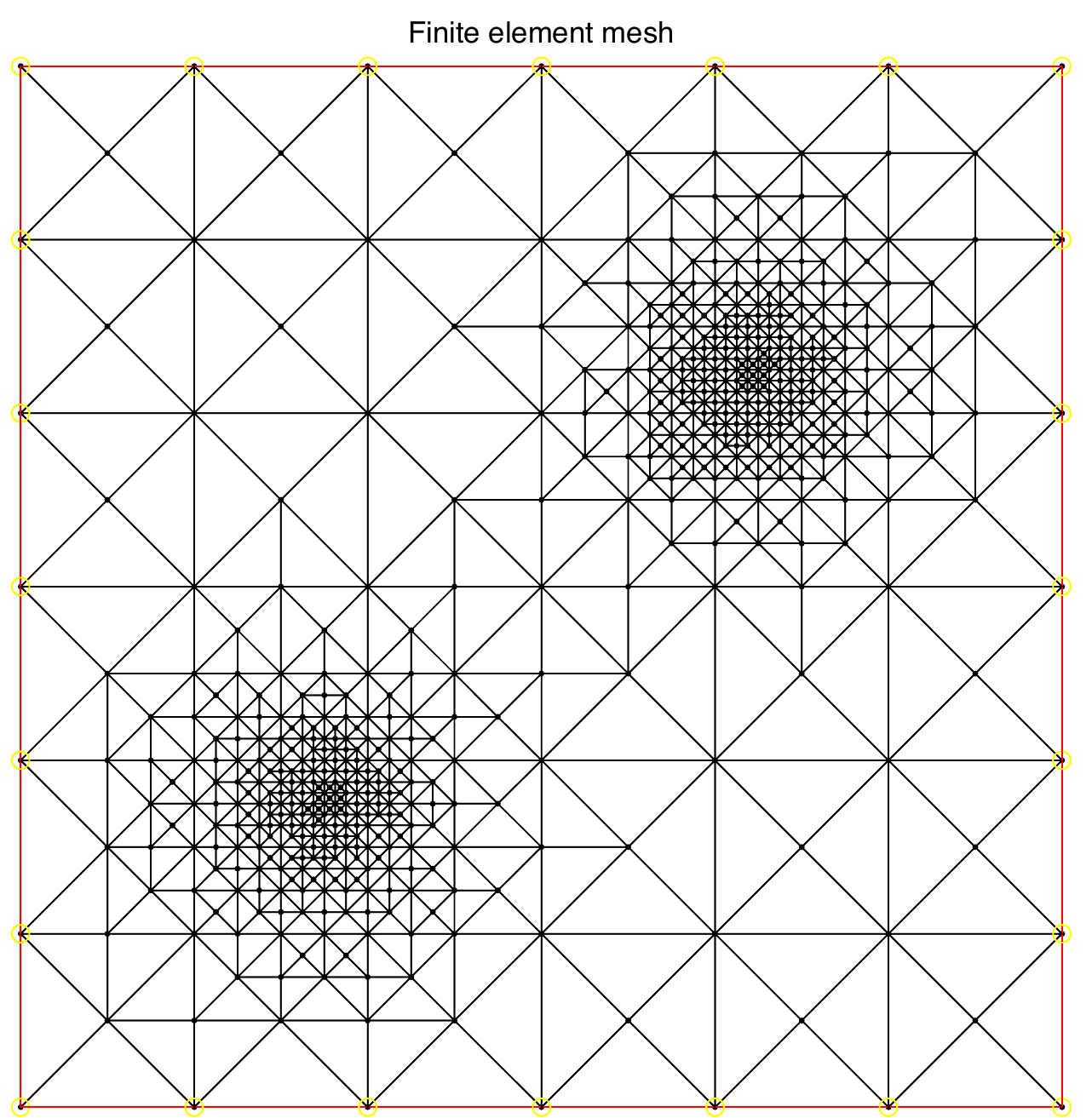}~
\includegraphics[width=0.3\textwidth]{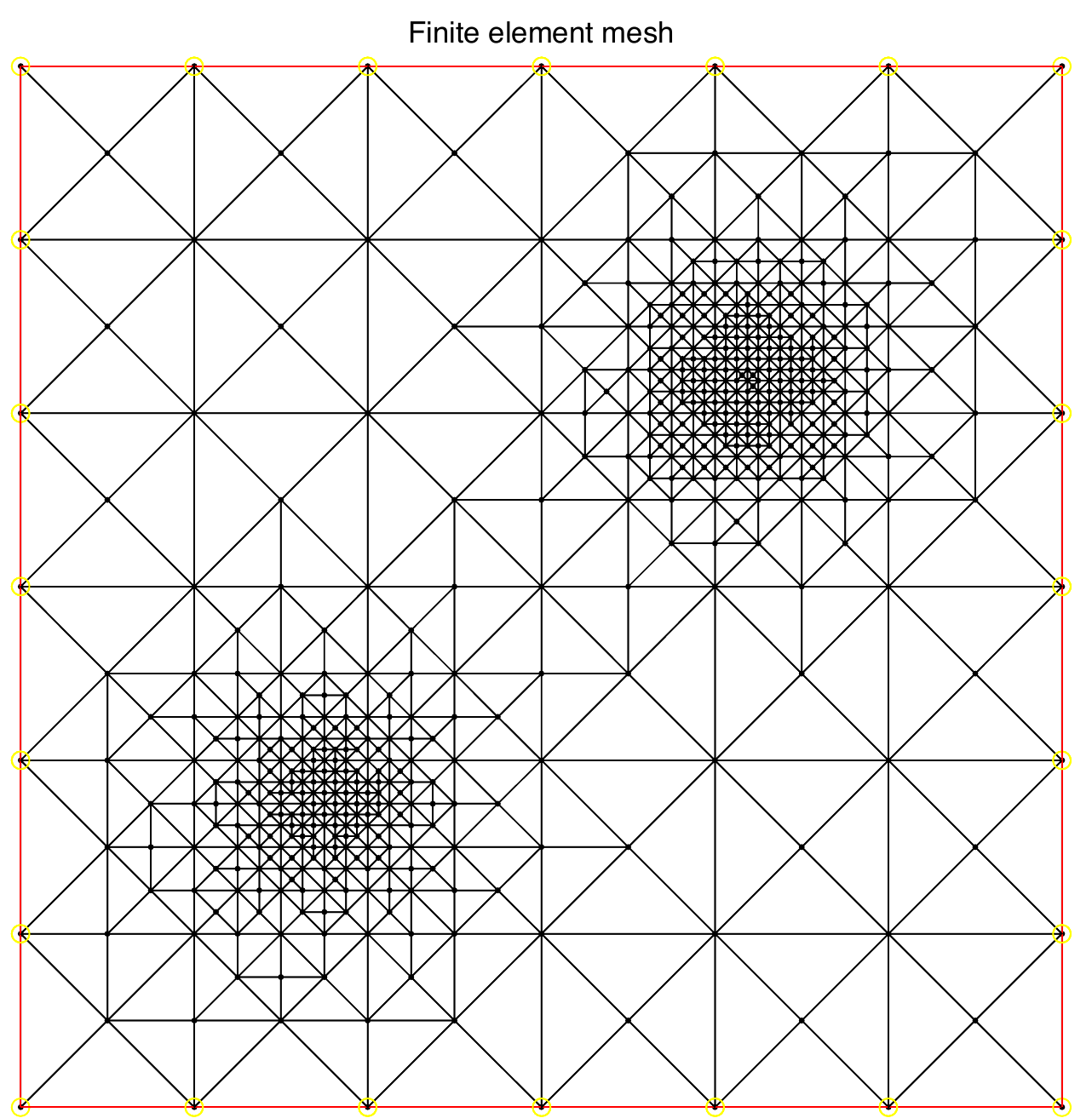}~
\includegraphics[width=0.3\textwidth]{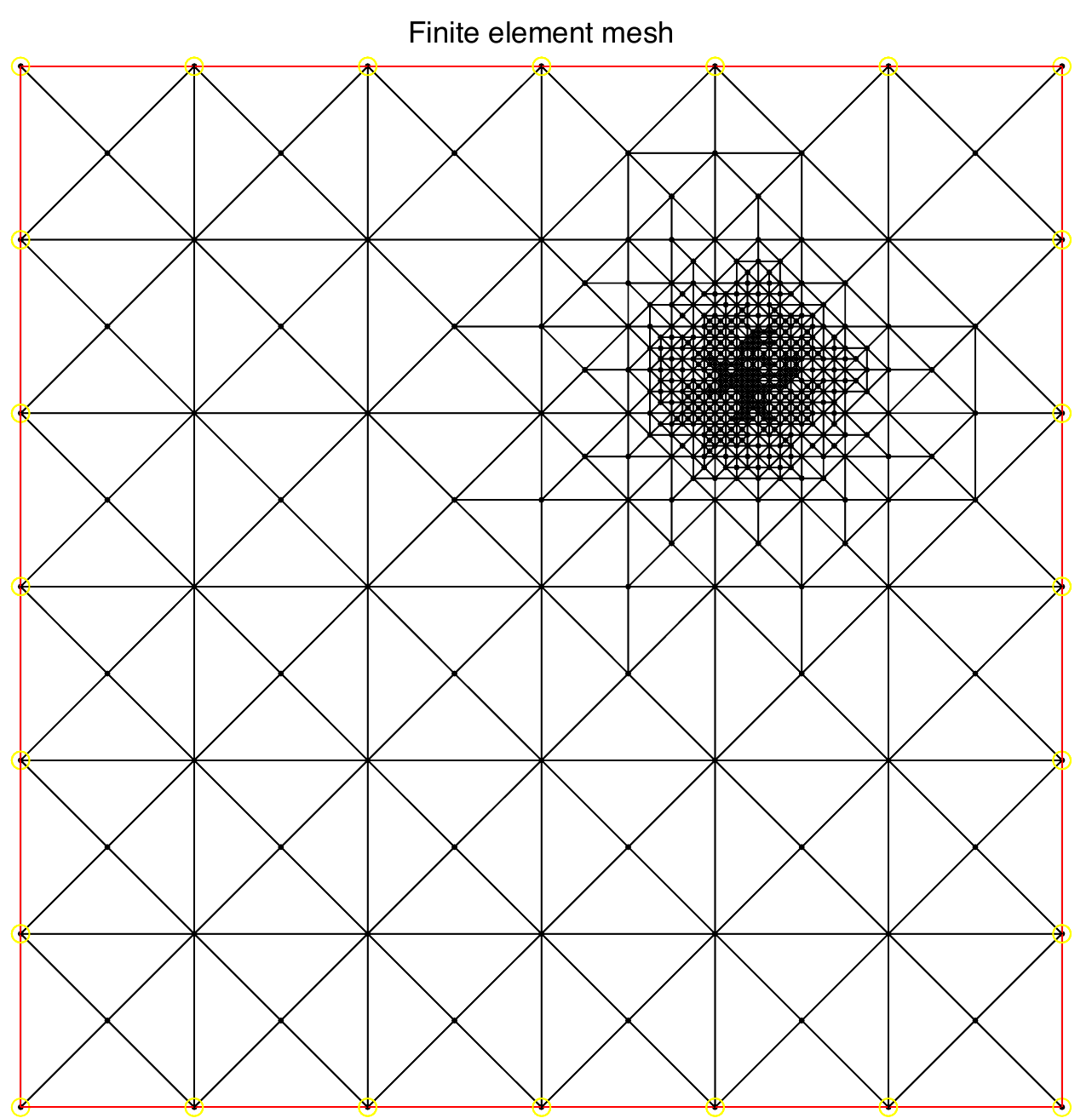}~
\caption{Left: 13 iterations of HPZ (1288 elements). Center: 26 iterations of MS (1162) elements). Right: 14 iterations of DWR (1134 elements) for parameter set~\eqref{HSD_433}.  }
\label{fig:mesh_433}
\end{figure}

In Figure~\ref{fig:HSD_233}, we locate the remote bump of the primal data at $(0.3, 0.3)$ which is far enough away from the spike in the goal function, so that its influence is minimal on $g(u)$. The three methods show a nearly identical error reduction rate when $a = 400$ in parameter set~\eqref{HSD_433}, while the residual based methods show a slight advantage when $a = 200$ in parameter set~\eqref{HSD_233}. 

 Figure~\ref{fig:mesh_433} shows the resulting adaptive meshes produced by difference algorithms for parameter set~\eqref{HSD_433}.  Even where the three methods produce nearly identical error reduction, the adaptive meshes are qualitatively different: DWR focuses on the interaction between the primal data and dual solution, HPZ and MS focus on the primal and dual data; however, HPZ has more concentrated refinement at the center of each region than does MS.

\begin{example}[Goal function with two spikes]
\end{example}
In this example, we consider the problem with a single spike in the primal data and a goal function $g(x,y)$ consisting of a Gaussian average about two separated points.  We keep the far point fixed and move the second point close to the spike in the primal data to investigate which of the algorithms are more effective as we vary the overlap of the refinement sets based on the primal and dual problems.  Compared to parameter sets~\eqref{HSD_266}-\eqref{HSD_433}, DWR generally fares as well or better than the residual based methods for~\eqref{HSD_w1x7}-\eqref{HSD_w2x4}. 

Here, we also manipulate $\omega$, the frequency of the sinusoid in the primal problem. We observe that varying the structure of the problem changes the relative efficiency of the three algorithms.

The goal function is given by
 \begin{align*}
g(x,y) = a \exp(-a((x-x_0)^2 + (y - y_0)^2)) + a\exp(-a((x-x_1)^2 + (y - y_1)^2)),
\end{align*} 
with $a = 400$  and  $(x_0, y_0) = (0.7, 0.7)$.  

The data $f(x,y)$ is chosen so the exact solution $u(x,y)$ is given by
\[
u = \sin(\omega \pi x) \sin(\omega \pi y) \f {1}{ 2[(x - x_p)^2 + (y - y_p)^2] + 10^{-3}}
~\text{ with }~(x_p,y_p) = (0.3, 0.3).
\]

\begin{align}
Figure~\ref{fig:HSD_xy73}  && (x_1,y_1) = (0.7,0.3)   &&  \omega = 1 \label{HSD_w1x7}, \\
					   && (x_1,y_1) = (0.7,0.3)   &&  \omega = 2 \label{HSD_w2x7}. \\
Figure~\ref{fig:HSD_xy53}  && (x_1,y_1) = (0.55,0.3)  &&  \omega = 1 \label{HSD_w1x5}, \\
				           && (x_1,y_1) = (0.55,0.3)  &&  \omega = 2 \label{HSD_w2x5}.\\
Figure~\ref{fig:HSD_xy43} && (x_1,y_1) = (0.4,0.3)    && \omega = 1 \label{HSD_w1x4}, \\
				            && (x_1,y_1) = (0.4,0.3)    &&  \omega = 2 \label{HSD_w2x4}.
\end{align}

Compared to parameter sets~\eqref{HSD_266}-\eqref{HSD_433}, DWR generally fares as well or better than the residual based methods for~\eqref{HSD_w1x7}-\eqref{HSD_w2x4}. 
\begin{figure}[htp]
\includegraphics[width=0.45\textwidth]{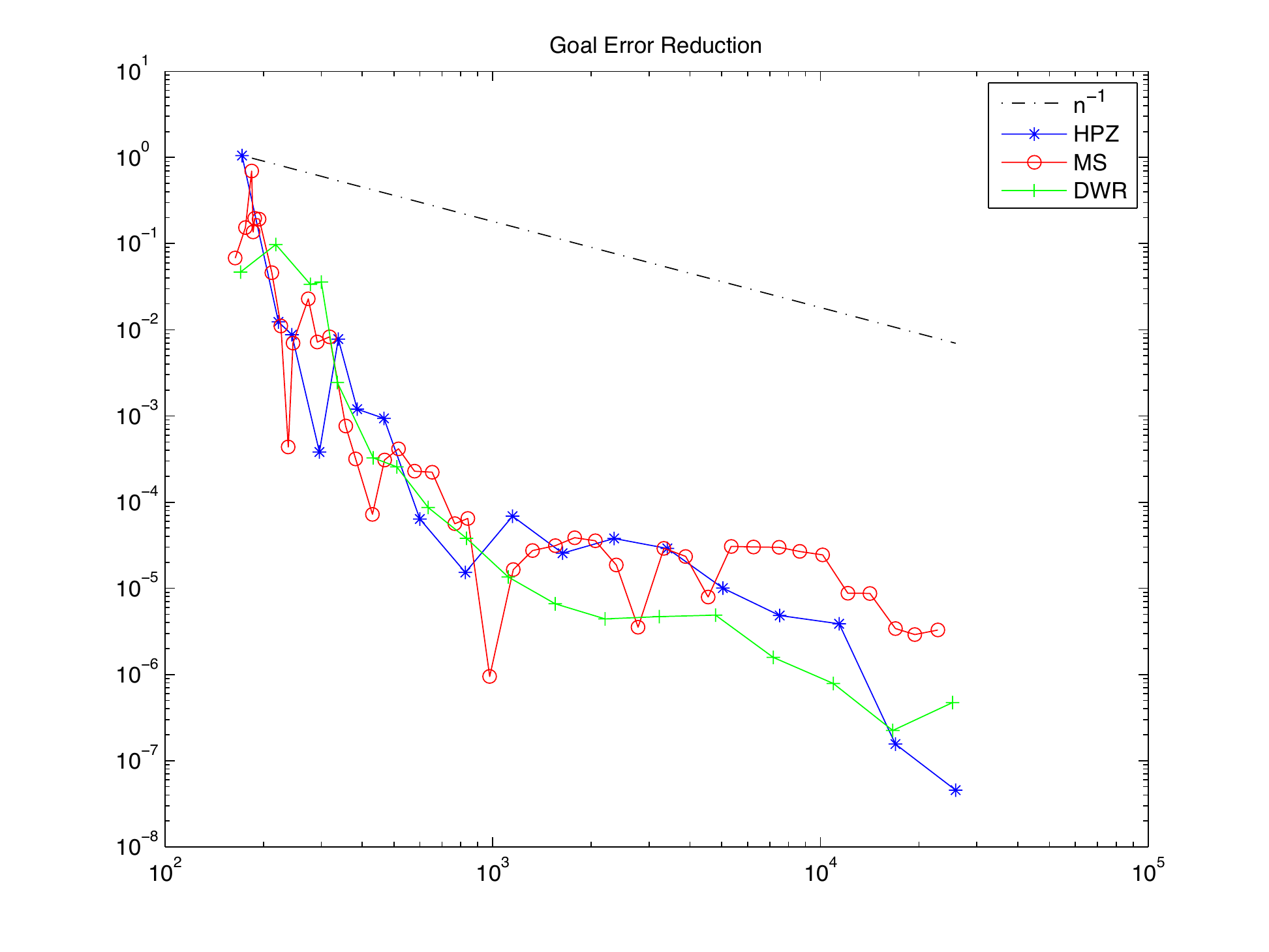}~
\includegraphics[width=0.45\textwidth]{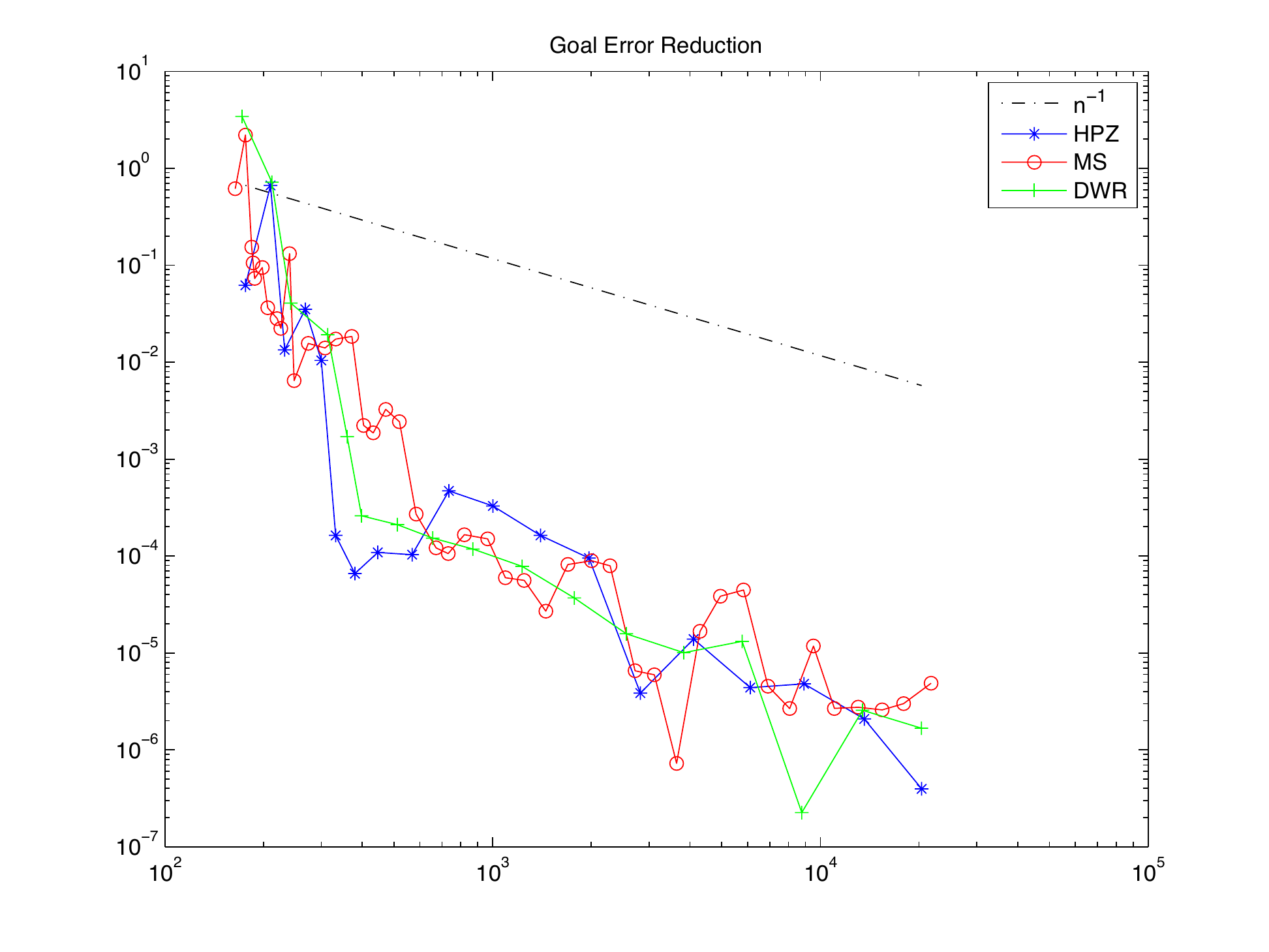}
\caption{Left: goal error after 19 HPZ, 46 MS and 19 DWR iterations for ~\eqref{HSD_w1x7}. 
Right: goal error after 20 HPZ, 47 MS and 18 DWR iterations for ~\eqref{HSD_w2x7}, compared with $n^{-1}$}.
\label{fig:HSD_xy73}
\end{figure}

\begin{figure}[htp]
\includegraphics[width=0.3\textwidth]{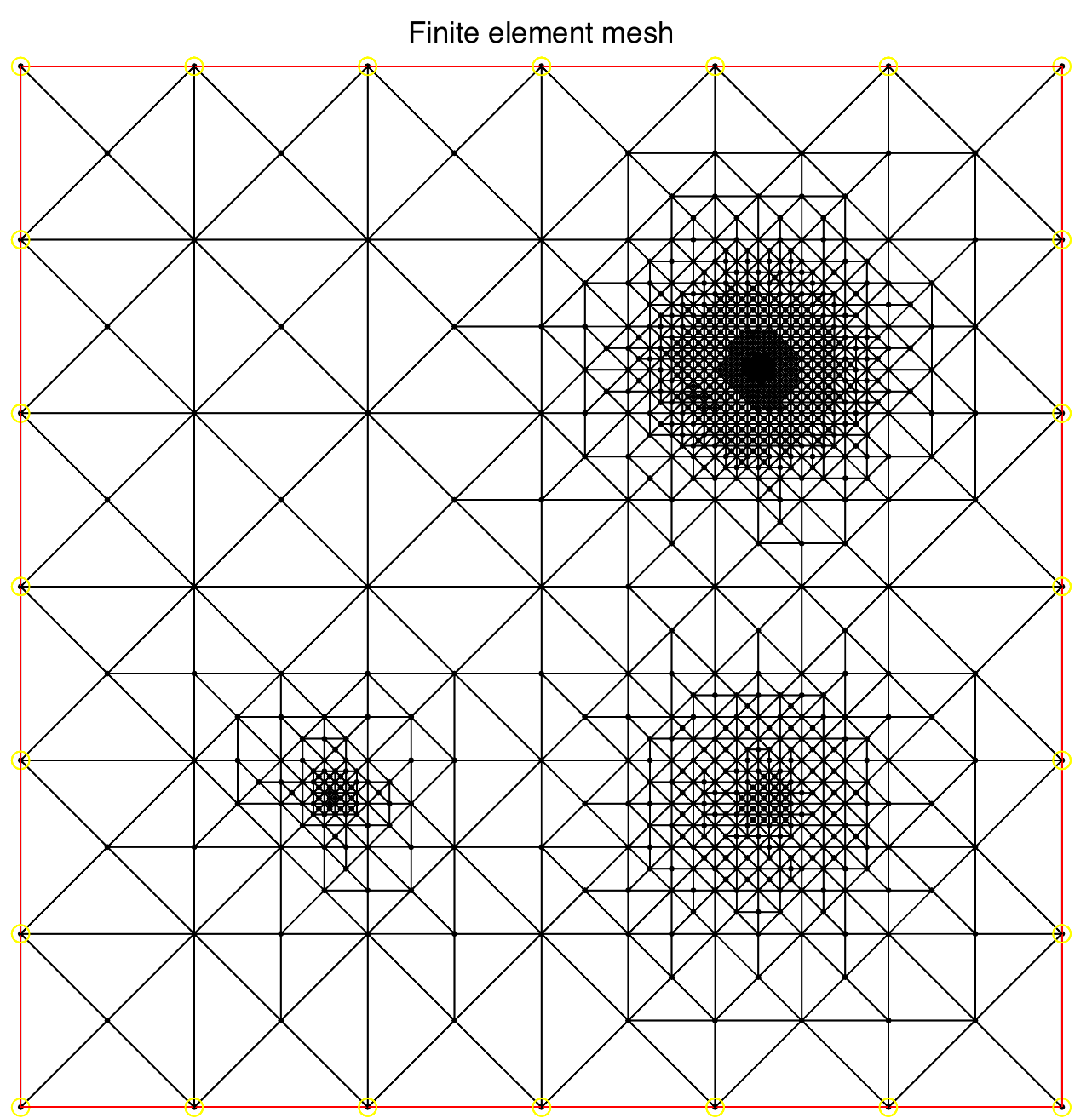}~
\includegraphics[width=0.3\textwidth]{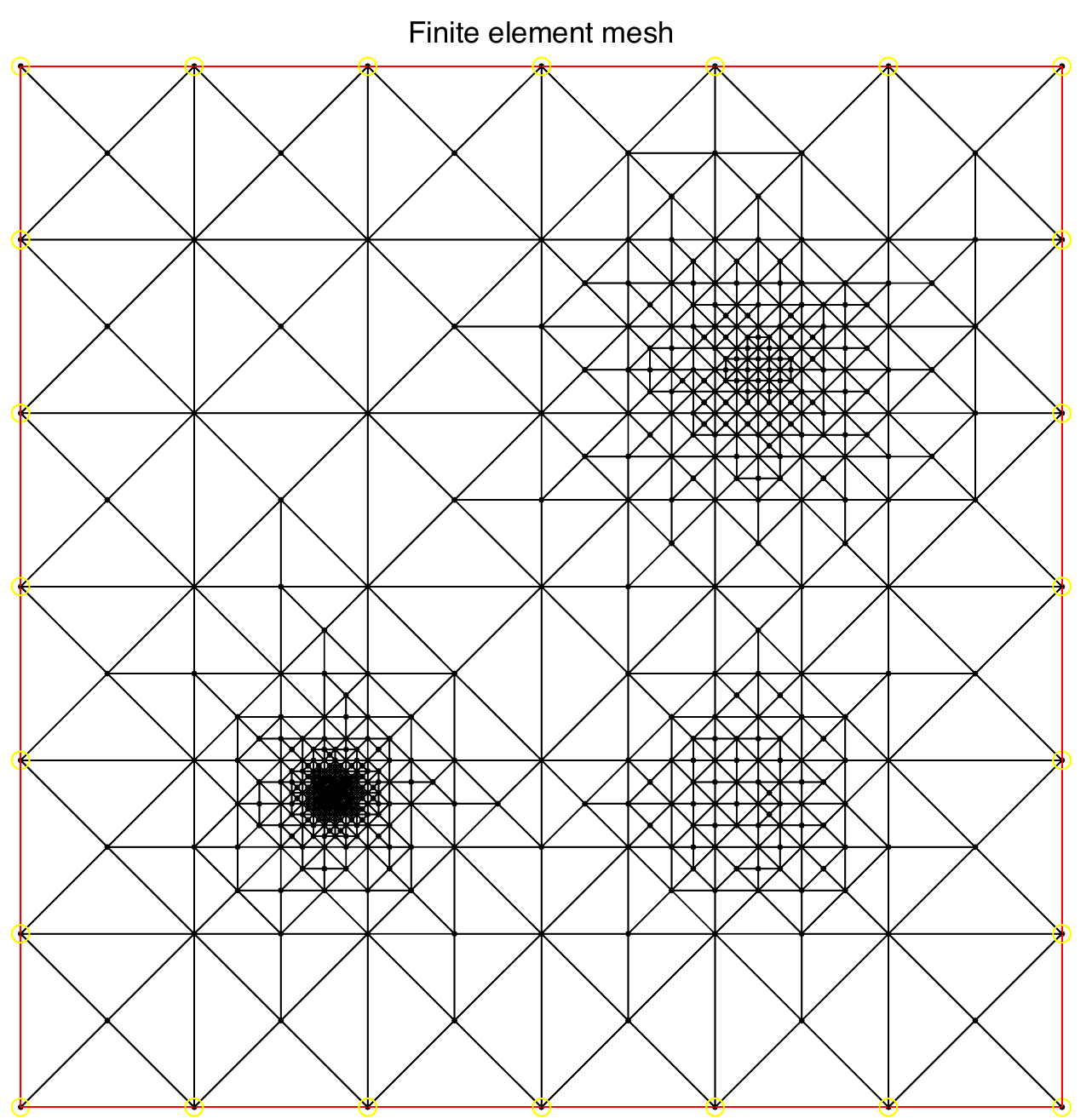}~
\includegraphics[width=0.3\textwidth]{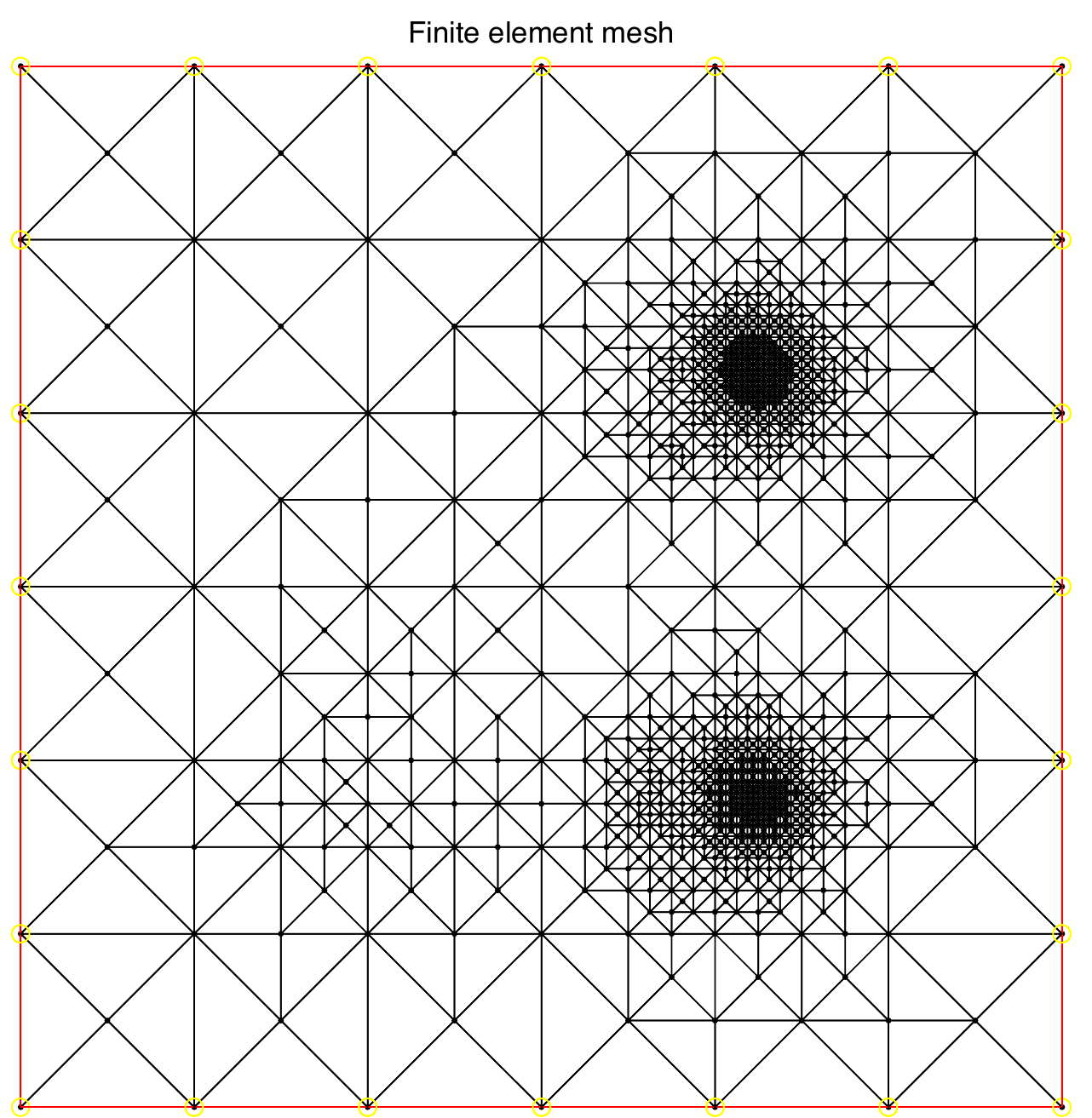}~
\caption{Left: 13 iterations of HPZ (1187 elements). Center: 26 iterations of MS (1156) elements). Right: 14 iterations of DWR (1115 elements) for ~\eqref{HSD_w1x7}.  }
\label{fig:mesh_w1x7}
\end{figure}
For the parameter sets~\eqref{HSD_w1x7} and~\eqref{HSD_w2x7} shown in Figure~\ref{fig:HSD_xy73} both dual spikes are remote from the primal data.  As seen in Figure~\ref{fig:mesh_w1x7} each algorithm displays a distinct trend in its adaptive refinement: HPZ refines for both primal and dual; MS refines for both with a bias towards the primal, with 17 primal refinements and 10 dual refinements in this case; DWR refines with a bias towards the dual data.   When $\omega=1$, HPZ and DWR show similar goal error reduction while MS stalls  at least on these relatively early refinements.  For $\omega = 2$, MS still shows a slight tendency to refine more for the primal than the dual; however the error reductions is generally similar to the other two methods. 

\begin{figure}[htp]
\includegraphics[width=0.45\textwidth]{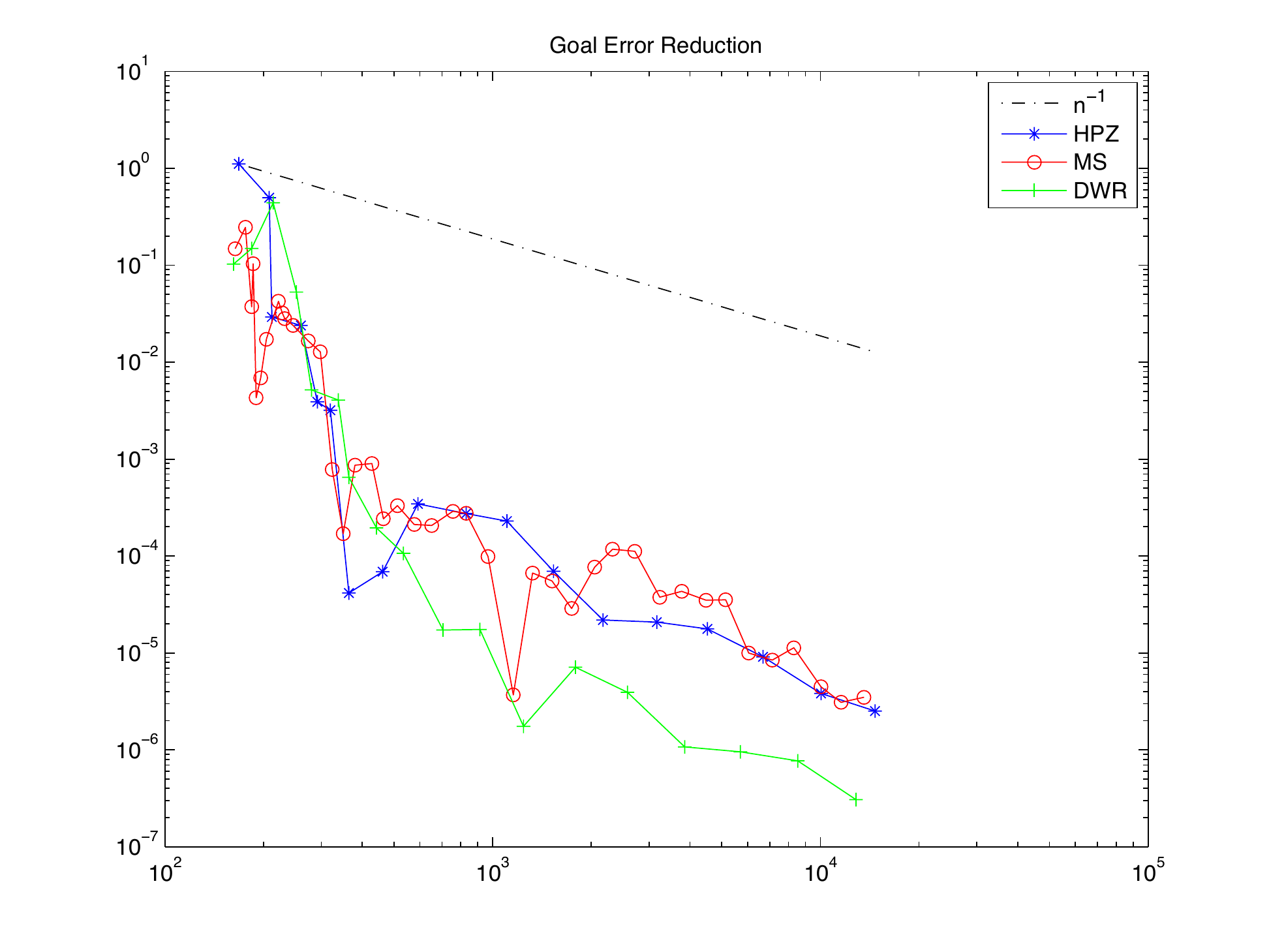}~
\includegraphics[width=0.45\textwidth]{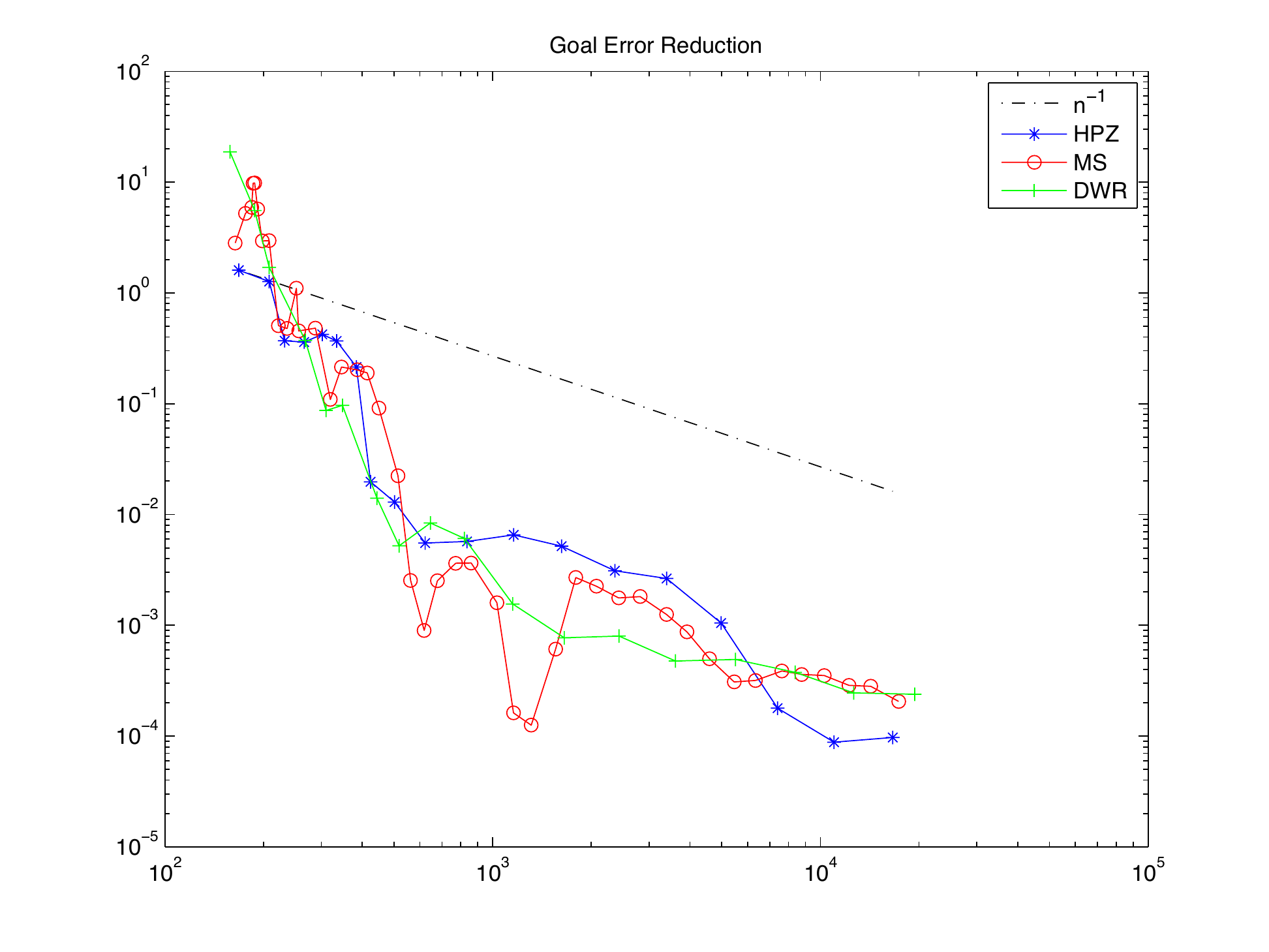}
\caption{Left: goal error after 19 HPZ, 44 MS and 19 DWR iterations for ~\eqref{HSD_w1x5}.  Right: goal error after 20 HPZ, 46 MS and 19 DWR iterations for ~\eqref{HSD_w2x5}, compared with $n^{-1}$. }
\label{fig:HSD_xy53}
\end{figure}

Figure~\ref{fig:HSD_xy53} shows the performance of the algorithms for parameter sets~\eqref{HSD_w1x5} and~\eqref{HSD_w2x5}. Here, the residual methods are similar and both outperformed by DWR in the case $\omega=1$ while all three methods are similar in the case $\omega=2$, with HPZ showing a trend towards slightly greater goal error reduction.

\begin{figure}[htp]
\includegraphics[width=0.45\textwidth]{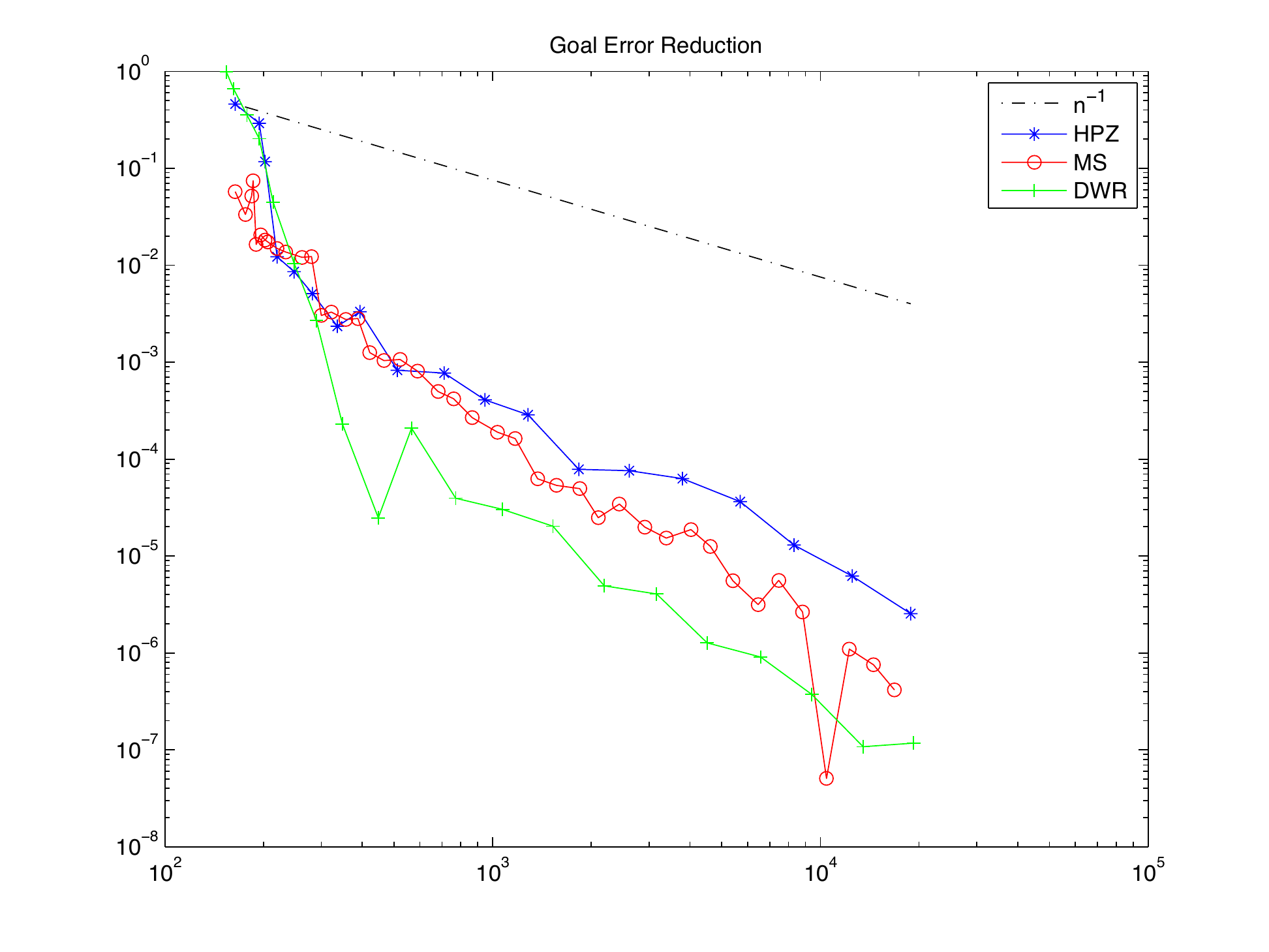}
\includegraphics[width=0.45\textwidth]{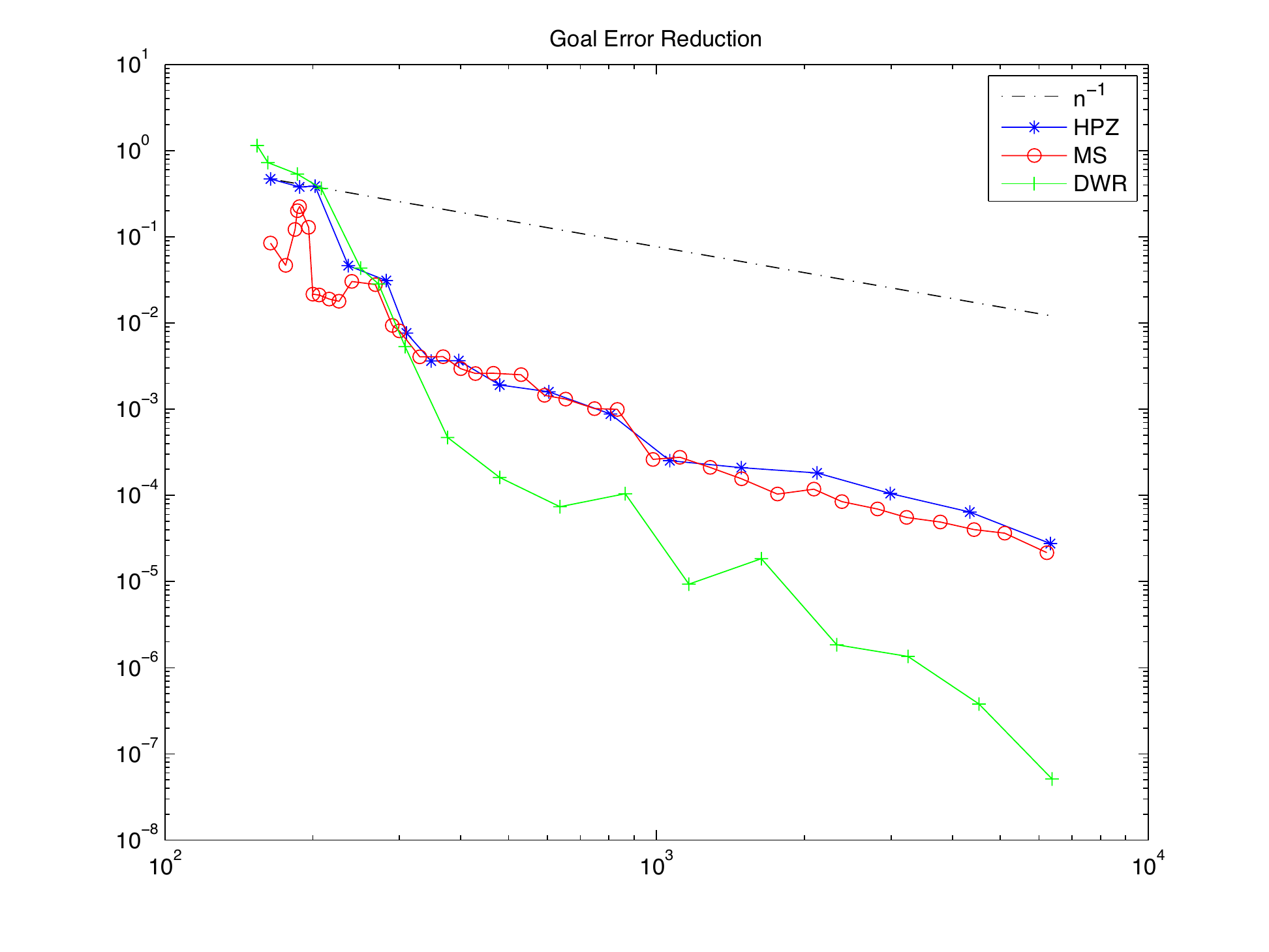}
\caption{Left: goal error after 20 HP, 45 MS and 21 DWR iterations for ~\eqref{HSD_w1x4}. 
Right:  goal error after 18 HP, 40 MS and 18 DWR iterations for ~\eqref{HSD_w2x4}, compared with $n^{-1}$. }
\label{fig:HSD_xy43}
\end{figure}

Finally, Figure~\ref{fig:HSD_xy43} shows the performance of these methods for parameter sets~\eqref{HSD_w1x4} and~\eqref{HSD_w2x4}.  In these examples, DWR outperforms the residual based methods.  In these two cases, the spikes in primal data and dual solution have an isolated area of overlap that coincides with the spikes in the primal solution and dual data, the situation that makes the DWR method the most efficient (\emph{cf.}~\cite{HoPo11a}).  Varying the frequency of the primal data again changes the relative efficiencies of the residual based methods.  In contrast to~\eqref{HSD_w1x7} and~\eqref{HSD_w2x7}, the performance of MS decreases when the frequency $\omega$ increases from $1$ to $2$.

The effectiveness of the DWR method is based on the assumption that $\langle R(u_h),z_h \rangle$ is a good predictor for the error $g(e_h)$.  This appears to work so long as rapidly changing gradients in the dual solution coincide spatially with spikes in $g(x,y)$, and the primal residual $R(u_h)$ captures sufficient information about the primal solution in the vicinity of the influence function $z_h$.  For an example of where the first condition fails, we refer to the linear convection-diffusion problem discussed in~\cite{HoPo11a}, and a demonstration of the second condition is~\eqref{HSD_266}. Under certain conditions, namely a confined region where  the spikes in primal data and dual solution overlap that coincides with the overlap in the spikes in the primal solution and dual data, the DWR methods outperforms the residual based methods.

In many cases, all three methods display similar performance, yet with qualitatively different adaptive mesh refinements.  The relative performances of HPZ and MS do appear to be dependent on the structure of the primal problem, however it is not clear at this stage how to predict which algorithm will yield a better reduction in goal error.  In problems where the HPZ and MS results appear similar, we note that MS takes considerably longer to run as it may require approximately twice as many total iterations of the algorithm where most of the runtime is spent on nonlinear solves. We further emphasize that the results here consider the error vs. mesh cardinality, not total degrees of freedom.  It is of further interest to compare the performance of DWR with the residual based methods using higher order finite elements for the dual and possibly primal problems. Determining classes of problems for which each method is best suited is currently under investigation by the present authors.

\section{Conclusion}

In this article we developed convergence theory for a class of 
goal-oriented adaptive finite element algorithms for second 
order semilinear elliptic equations.
We first introduced several approximate dual problems,
and briefly discussed the target problem class.
We then reviewed some standard facts concerning conforming finite element 
discretization and error-estimate-driven adaptive finite element methods 
(AFEM).  
We included a brief summary of \emph{a priori} estimates for semilinear 
problems, and then described goal-oriented variations of the standard
approach to AFEM (GOAFEM).
Following the recent work of Mommer-Stevenson and Holst-Pollock
for linear problems, we established contraction of GOAFEM for the
primal problem.
We also developed some additional estimates that make it possible
to establish contraction of the combined quasi-error, and showed
convergence in the sense of the quantity of interest.
Some simple numerical experiments confirmed these theoretical predictions and
demonstrated that our method performs comparably to other standard adaptive
goal-oriented strategies, and has the additional advantage of provable convergence 
for problems where the theory has not been developed for the other two methods.
Our analysis was based on the recent contraction frameworks for the 
semilinear problem developed by Holst, Tsogtgerel, and Zhu and 
Bank, Holst, Szypowski and Zhu and those for linear problems 
as in Cascon, Kreuzer, Nochetto and Siebert, and Nochetto, Siebert, and Veeser.
In addressing the goal-oriented problem we based our approach
on that of Mommer and Stevenson for symmetric linear problems and 
Holst and Pollock for nonsymmetric problems.
However, unlike the linear case, we were faced with 
tracking linearized and approximate dual sequences in order to 
establish contraction with respect to the quantity of interest.

In the present paper we assume the primal and approximate dual solutions are solved on the same mesh at each iteration.  The determination of strong convergence results for a method which solves the primal (nonlinear) problem on a coarse mesh and the dual on a fine mesh is the subject of future investigation.
\section*{Acknowledgments}
   \label{sec:ack}

MH was supported in part by NSF Awards~1065972, 1217175, 1262982, 1318480, 
and by AFOSR Award FA9550-12-1-0046.
SP and YZ were supported in part by NSF Awards~1065972 and 1217175.
YZ was also supported in part by NSF DMS 1319110, and in part by University Research Committee Grant No. F119 at Idaho State University, Pocatello, Idaho.

\bibliographystyle{abbrv}
\bibliography{semiA_refs,../../bib/library}

\vspace*{0.5cm}

\end{document}